\documentclass[a4paper,12pt,draft]{amsart}
\usepackage[margin=30mm]{geometry}
\usepackage{amssymb, amsmath, amsthm, amscd}

\usepackage[T1]{fontenc}
\usepackage{eucal,mathrsfs,dsfont}
\usepackage{color}

\renewcommand{\le}{\leqslant}
\renewcommand{\ge}{\geqslant}

\definecolor{mno}{rgb}{0.5,0.1,0.5}

\newcommand{\R}{\mathds R}

\newcommand{\Pp}{\mathds P}
\newcommand{\Ee}{\mathds E}

\newcommand{\I}{\mathds 1}

\newcommand{\D}{\mathscr{D}}
\newcommand{\Nn}{\mathcal{N}}

\newtheorem{theorem}{Theorem}[section]
\newtheorem{lemma}[theorem]{Lemma}
\newtheorem{proposition}[theorem]{Proposition}

\theoremstyle{definition}

\newtheorem{example}[theorem]{Example}
\newtheorem{remark}[theorem]{Remark}

\begin{document}
\allowdisplaybreaks
\title[Intrinsic Ultracontractivity of Feynman-Kac
Semigroups] {\bfseries Intrinsic  Ultracontractivity of Feynman-Kac
Semigroups for Symmetric Jump Processes}
\author{Xin Chen\qquad Jian Wang}
\thanks{\emph{X.\ Chen:}
   Department of Mathematics, Shanghai Jiao Tong University, 200240 Shanghai, P.R. China. \texttt{chenxin\_217@hotmail.com}}
  \thanks{\emph{J.\ Wang:}
   School of Mathematics and Computer Science, Fujian Normal University, 350007 Fuzhou, P.R. China. \texttt{jianwang@fjnu.edu.cn}}

\date{}

\maketitle

\begin{abstract}
Consider a
symmetric non-local Dirichlet form $(D,\D(D))$ given by
\begin{equation*}
\begin{split}
& D(f,f)=\int_{\R^d}\int_{\R^d}\big(f(x)-f(y)\big)^2 J(x,y)\,dx\,dy
\end{split}
\end{equation*}
with $\D(D)$ the closure of the set of $C^1$ functions on $\R^d$
with compact support under the norm $\sqrt{D_1(f,f)}$, where
$D_1(f,f):=D(f,f)+\int f^2(x)\,dx$  and $J(x,y)$ is a nonnegative
symmetric measurable function on $\R^d\times \R^d$. Suppose that
there is a Hunt process $(X_t)_{t\ge 0}$ on $\R^d$ corresponding to
$(D,\D(D))$, and that $(L,\D(L))$ is its infinitesimal generator. We
study the intrinsic ultracontractivity for the Feynman-Kac semigroup
$(T_t^V)_{t\ge 0}$ generated by $L^V:=L-V$, where $V\ge 0$ is a
non-negative locally bounded measurable function such that Lebesgue
measure of the set $\{x\in \R^d: V(x)\le r\}$ is finite for every
$r>0$. By using intrinsic super Poincar\'{e} inequalities and
establishing an explicit lower bound estimate for the ground state,
we present general criteria for the intrinsic ultracontractivity of
$(T_t^V)_{t\ge 0}$. In particular, if
$$J(x,y)\asymp|x-y|^{-d-\alpha}\I_{\{|x-y|\le 1\}}+e^{-|x-y|^\gamma}\I_{\{|x-y|> 1\}}$$  for some $\alpha \in (0,2)$ and $\gamma\in(1,\infty]$,
and the potential function $V(x)=|x|^\theta$ for some $\theta>0$,
then $(T_t^V)_{t\ge 0}$ is intrinsically ultracontractive if and
only if $\theta>1$. When $\theta>1$, we have the following explicit
estimates for the ground state $\phi_1$
$$c_1\exp\Big(-c_2 \theta^{\frac{\gamma-1}{\gamma}}|x|
\log^{\frac{\gamma-1}{\gamma}}(1+|x|)\Big) \le \phi_1(x) \le
c_3\exp\Big(-c_4 \theta^{\frac{\gamma-1}{\gamma}}|x|
\log^{\frac{\gamma-1}{\gamma}}(1+|x|)\Big) ,$$ where $c_i>0$
$(i=1,2,3,4)$ are constants. We stress that our method efficiently applies to the Hunt process
$(X_t)_{t \ge 0}$ with finite range jumps, and some irregular
potential function $V$ such that $\lim_{|x| \to
\infty}V(x)\neq\infty$.
\medskip

\medskip

\end{abstract}
\allowdisplaybreaks

\section{Introduction and Main Results}\label{section1}
\subsection{Setting and assumptions}\label{subsection1-1}
Let $(D,\D(D))$ be a symmetric non-local Dirichlet form as follows
\begin{equation}\label{non-local}
\begin{split}
 D(f,f)&=\int_{\R^d}\int_{\R^d}\big(f(x)-f(y)\big)^2 J(x,y)\,dx\,dy,\\
\D(D)&=\overline{C_c^{1}(\R^d)}^{D_1},
\end{split}
\end{equation}
where $J(x,y)$ is a non-negative measurable function on $\R^d\times
\R^d$ satisfying that

\begin{itemize}
\item[(1)] $J(x,y)=J(y,x)$ for all $x$, $y\in\R^d$;

\item[(2)] There exist $\alpha_1,\alpha_2\in(0,2)$
with $\alpha_1\le \alpha_2$ and positive $\kappa, c_1,c_2$ such that
\begin{equation}\label{e3-1}
c_1|x-y|^{-d-\alpha_1}\le J(x,y)\le c_2|x-y|^{-d-\alpha_2}, \quad
0<|x-y|\le \kappa
\end{equation} and
\begin{equation}\label{e3-2}
\sup_{x \in \R^d}\int_{\{|x-y|>\kappa\}}J(x,y)\,dy<\infty.
\end{equation}
\end{itemize}
Here, $C_c^{1}(\R^d)$ denotes the space of $C^1$ functions on $\R^d$
with compact support, $D_1(f,f):=D(f,f)+\int f^2(x)\,dx$ and $\D(D)$
is the closure of $C_c^1(\R^d)$ with respect to the metric
$D_1(f,f)^{1/2}.$ It is easy to see that \eqref{e3-1} and
\eqref{e3-2} imply that
\begin{equation*}\label{con1-1}
\sup_{x\in\R^d}\int_{\R^d}(1\wedge|x-y|^2)J(x,y)\,dy<\infty,
\end{equation*}
which in turn gives us that $D(f,f)<\infty$ for each $f \in
C_c^{1}(\R^d)$. According to \cite[Example 1.2.4]{FOT}, we know that
$(D,\D(D))$ is a regular Dirichlet form on $L^2(\R^d;dx)$. Therefore,
there exists $\mathcal{N}\subset\R^d$ having zero capacity with
respect to the Dirichlet form $(D,\D(D))$, and there is a Hunt
process $\big((X_t)_{t \ge 0},\Pp^x\big)$ with state space
$\R^d\backslash\Nn$ such that for every $f\in L^2(\R^d;dx)$ and
$t>0$, $x\mapsto \Ee^x(f(X_t))$ is a quasi-continuous version of
$T_tf$, where $\Ee^x$ is the expectation under the probability
measure $\Pp^x$ and $(T_t)_{t\ge0}$ is the $L^2$-semigroup
associated with $(D,\D(D))$, see e.g. \cite[Theorem 1.1]{BBCK}. The set $\Nn$ is called the properly exceptional set of the
process $(X_t)_{t\ge0}$ (or, equivalently, of the Dirichlet form
$(D,\D(D)))$, and it has zero Lebesgue measure. Furthermore, by (\ref{e3-1}), (\ref{e3-2}) and
the proof of \cite[Theorem 1.2]{BBCK}, there exists a positive
symmetric measurable function $p(t,x,y)$ defined on
$[0,\infty)\times(\R^d \setminus \Nn )\times (\R^d \setminus \Nn)$
such that
\begin{equation*}
T_f(x)=\Ee^x\big(f(X_t)\big)=\int_{\R^d \setminus
\Nn}p(t,x,y)f(y)\,dy,\quad x \in \R^d \setminus\Nn,\ t>0,\ f \in
B_b(\R^d);
\end{equation*}
moreover, for every $t>0$ and $y\in \R^d \setminus \Nn$, the
function $x\mapsto p(t,x,y)$ is quasi-continuous on $\R^d
\setminus\Nn$, and for any $t>0$ there is a constant $c_t>0$ such
that for any $x$, $y\in\R^d \setminus\Nn$, $0<p(t,x,y)\le c_t$.

First, we make the following continuity assumption on $p(t,x,y)$.

\begin{itemize}
\item[{\bf (A1)}] \emph{$\Nn=\varnothing$. For every $t>0$, the function $(x,y)\mapsto p(t,x,y)$ is
continuous on $\R^d \times \R^d$, and $0<p(t,x,y)\le c_t$ for all $x, y \in \R^d$.}
\end{itemize}
In particular, {\bf (A1)} implies that the Hunt process
$\big((X_t)_{t\ge0}, \Pp^x\big)$ is well defined for all $x\in\R^d$,
and the associated strongly continuous Markov semigroup $(T_t)_{t
\ge 0}$ is ultracontractive, i.e. $\|T_tf\|_{L^\infty(\R^d;dx)} \le
c_t\|f\|_{L^1(\R^d;dx)}$ for all $t>0$ and every $f\in
L^1(\R^d;dx)$.

When for any $x,y\in\R^d$, $J(x,y)=\rho(x-y)$ holds with some non-negative
measurable function $\R^d$ such that $\rho(z)=\rho(-z)$ for all $z\in\R^d$ and
$\int_{\R^d\setminus\{0\}}(1\wedge|z|^2)\rho(z)\,dz<\infty$,
the corresponding Hunt process $(X_t)_{t \ge 0}$ is a symmetric
L\'evy process having L\'{e}vy jump measure $\nu(dz):=\rho(z)\,dz$. In
this case, assumption {\bf (A1)} is equivalent to
$e^{-t\Psi_0(\cdot)}\in L^1(\R^d;dx)$ for any $t>0$, where the
characteristic exponent or the symbol $\Psi_0$ of L\'{e}vy process
$(X_t)_{t \ge 0}$ is defined by
$$
    \Ee^x\bigl(e^{i\langle{\xi},{X_t-x}\rangle}\bigr)
    =e^{-t\Psi_0(\xi)},\quad x,\xi\in\R^d, t>0.
$$ It is well known that the L\'{e}vy process enjoys the space-homogeneous property.
For sufficient conditions on the jump density $J(x,y)$ such that the
associated space-inhomogeneous Hunt process $(X_t)_{t \ge 0}$
satisfies assumption {\bf (A1)}, we refer the reader to \cite{CK,
CK1, CKK2,BBCK,CKK1} and the references therein.

\ \

Let $V$ be a non-negative measurable and locally bounded potential
function on $\R^d$. Define the Feynman-Kac semigroup
$(T^V_t)_{t\ge0}$ associated with the Hunt process $(X_t)_{t \ge 0}$ as
follows:
\begin{equation*} T^V_t(f)(x)=\Ee^x\left(\exp\Big(-\int_0^tV(X_s)\,ds\Big)f(X_t)\right),\,\, x\in\R^d,
f\in L^2(\R^d;dx).\end{equation*} It is easy to check that
$(T_t^V)_{t \ge 0}$ is a bounded symmetric semigroup on
$L^2(\R^d;dx)$.
Furthermore, following the arguments in \cite[Section 3.2]{CZ} (see
also the proof of \cite[Lemma 3.1]{KS}), we can find that for each
$t>0$, $T_t^V$ is a bounded operator from $L^1(\R^d;dx)$ to
$L^{\infty}(\R^d;dx)$, and there exists a bounded, positive and
symmetric transition kernel $p^V(t,x,y)$ on $[0,\infty)\times\R^d\times \R^d$ such
that for any $t>0$, the function $(x,y)\mapsto
p^V(t,x,y)$ is continuous on $\R^d \times \R^d$, and for every $1 \le p \le \infty$,
\begin{equation*}
T_t^Vf(x)=\int_{\R^d} p^V(t,x,y)f(y)\,dy,\quad x\in \R^d, f\in L^p(\R^d;dx).
\end{equation*}

\ \

The following result gives us an easy criterion for the compactness of
the semigroup $(T_t^V)_{t\ge0}$. The proof is mainly based on \cite[Corollary 1.3]{WW08}.
For the sake of completeness, we will provide its proof in the Appendix.

\begin{proposition}\label{p1-1}
Under Assumption {\bf (A1)}, if for any $r>0$, Lebesgue measure of the set $$\{x\in\R^d: V(x)\le r\}$$ is finite, then the semigroup
$(T^V_t)_{t\ge0}$ is compact.
\end{proposition}

From now on, we will take the following assumption:

\begin{itemize}
\item[{\bf (A2)}]  \emph{Lebesgue measure of the set $\{x\in\R^d: V(x)\le r\}$ is finite for any $r>0$. }
\end{itemize}
In particular, according to Proposition \ref{p1-1}, the semigroup
$(T^V_t)_{t\ge0}$ is compact. By general theory of semigroups for
compact operators, there exists an orthonormal basis of eigenfunctions $\{\phi_n\}_{n=1}^\infty$
in
$L^2(\R^d;dx)$
associated with corresponding eigenvalues
$\{\lambda_n\}_{n=1}^\infty$ satisfying $0< \lambda_1<\lambda_2\le
\lambda_3\cdots$ and $\lim_{n\to\infty}\lambda_n=\infty$. That is,
$L^V \phi_n=-\lambda_n \phi_n$ and $T_t^V\phi_n=e^{-\lambda_n
t}\phi_n$, where $(L^V,\D(L^V))$ denotes the infinitesimal
generator of the semigroup $(T_t^V)_{t \ge 0}$.
The first eigenfunction $\phi_1$ is called ground state in the
literature. Furthermore, according to assumptions above, we have the
following property for $\phi_1$. The proof is also left to the Appendix.

\begin{proposition}\label{p1-2}
Under Assumptions {\bf (A1)} and {\bf (A2)}, there exists a version
of $\phi_1$ which is bounded, continuous and strictly positive.
\end{proposition}

\ \

To derive a upper bound estimate for the ground state $\phi_1$, we
need the explicit expression of the operator $L^V$, which is  given
by
$$L^Vf(x)=Lf(x)-V(x)f(x).$$ Here, $(L,\D(L))$ is the generator associated with Dirichlet form $(D,\D(D))$. In
L\'{e}vy case, it is easy to see that for any $f\in
C_c^2(\R^d)\subset\D(L)$,
$$
Lf(x)=\int_{\R^d} \Big(f(x+z)-f(x)-\langle \nabla f(x),z \rangle \I_{\{|z|\le 1\}}\Big)\rho(z)\,dz,$$ where $\rho$ is the density function of the L\'{e}vy measure.
For general non-local Dirichlet form $(D,\D(D))$, if for every $x \in
\R^d$,
\begin{equation*}\label{oper-1}
\int_{\{|z|\le 1\}} |z| \left|J(x,x+z)-J(x,x-z)\right|\,dz<\infty,
\end{equation*} and for any $r>0$ large enough, \begin{equation*}\label{oper-2} x\mapsto \I_{B(0,2r)^c}
\int_{\{|x+z|\le r\}} J(x,x+z)\,dz \in L^2(\R^d;dx),\end{equation*}
then $C_c^2(\R^d)\subset \D(L)$ and for any $f\in C_c^2(\R^d)$,
\begin{equation*}
\begin{split}
Lf(x)=&
\int_{\R^d} \Big(f(x+z)-f(x)-\langle \nabla f(x),z \rangle \I_{\{|z|\le 1\}}\Big)J(x,x+z)\,dz\\
&+\frac{1}{2}\int_{\{|z|\le 1\}}\langle\nabla f(x),
z\rangle\left(J(x,x+z)-J(x,x-z)\right)\,dz,
\end{split}
\end{equation*} e.g.\ see \cite[Theorem 1.1]{W2009} for more details.
According to the discussions above, sometime we adopt the following
regular assumptions on $J(x,y)$ and the operator $L^V$, which are
satisfied for all symmetric L\'{e}vy processes.
\begin{itemize}
\item[{\bf (A3)}]  \emph{The jump kernel $J(x,y)$ satisfies that $$\sup_{x\in\R^d}\int_{\{|z|\le 1\}} |z|
\left|J(x,x+z)-J(x,x-z)\right|\,dz<\infty,$$ and for any $f\in C_c^2(\R^d)\subseteq \D(L^V)$,
\begin{equation}\label{ope11}
\begin{split}
L^Vf(x)=&
\int_{\R^d} \Big(f(x+z)-f(x)-\langle \nabla f(x),z \rangle \I_{\{|z|\le 1\}}\Big)J(x,x+z)\,dz\\
&+\frac{1}{2}\int_{\{|z|\le 1\}}\langle\nabla f(x),
z\rangle\left(J(x,x+z)-J(x,x-z)\right)\,dz-V(x)f(x).
\end{split}
\end{equation}
 }
\end{itemize}

\ \

\subsection{Main results}

\emph{Throughout this paper, we always assume that assumptions {\bf
(A1)} and {\bf (A2)} hold, and that the ground state $\phi_1$ is
bounded, continuous and strictly positive.} In this paper, we are
concerned with the intrinsic ultracontractivity for the semigroup
$(T_t^V)_{t\ge0}$. We first recall the definition of intrinsic
ultracontractivity for Feynman-Kac semigroups introduced in
\cite{DS}. The semigroup $(T_t^V)_{t\ge0}$ is intrinsically
ultracontractive if and only if for any $t>0$, there exists a
constant $C_t>0$ such that for all $x$, $y\in\R^d$,
\begin{equation*}\label{iuc}p^V(t,x,y)\le C_t\phi_1(x)\phi_1(y).\end{equation*}
In the framework of the semigroup theory, define
\begin{equation}\label{e1}
\tilde{T}_t^Vf(x):=\frac{e^{\lambda_1t}}{\phi_1(x)}T_t^V((\phi_1f))(x),\quad t>0,
\end{equation}
which is a Markov semigroup on $L^2(\R^d; \phi^2_1(x)\,dx)$. Then, $(T_t^V)_{t\ge0}$ is intrinsically ultracontractive if and only if $(\tilde{T}_t^V)_{t\ge0}$ ultracontractive, i.e., for every $t>0$,
$\tilde{T}_t^V$ is a bounded operator from $L^2(\R^d; \phi^2_1(x)\,dx)$ to $L^\infty(\R^d; \phi^2_1(x)\,dx)$.

Recently, the intrinsic ultracontractivity of $(T_t^V)_{t\ge0}$
associated with some special pure jump symmetric L\'evy process
$(X_t)_{t\ge0}$ has been investigated in \cite{KS,KK,KL}. The
approach of all these cited papers is based on sharp and explicit
pointwise upper and lower bound estimates for the ground state $\phi_1$
corresponding to the semigroup $(T_t^V)_{t\ge0}$. However, to apply
such powerful technique,  some restrictions on the density function
of jump kernel are needed, e.g.\ see \cite[Assumption 2.1]{KL}. In
particular, in L\'{e}vy case the following typical example
\begin{equation}\label{jjj} J(x,y)\asymp|x-y|^{-d-\alpha}\I_{\{|x-y|\le
1\}}+e^{-|x-y|^\gamma}\I_{\{|x-y|> 1\}}\end{equation} with $\alpha \in (0,2)$
and $\gamma\in (1,\infty]$ is not included in \cite{KS,KL,KK}. Here and in what follows,
for two functions
$f$ and $g$ defined on $\R^d\times \R^d$, $f \asymp g$ means that there is a constant $c>1$ such that $c^{-1}g(x,y)\le f(x,y) \le c g(x,y)$ for all $(x,y)\in\R^d\times \R^d$.
In particular, when $\gamma=\infty$,
$$J(x,y)\asymp|x-y|^{-d-\alpha}\I_{\{|x-y|\le 1\}},$$ which is
associated with the truncated symmetric $\alpha$-stable process. As
mentioned in \cite{CKK2, CKK1,BBCK}, such jump density function
$J(x,y)$ is very important in applications,
and
it
arises in statistical physics to model turbulence as well as in
mathematical finance to model stochastic volatility.

Furthermore, the following growth condition on the potential function
\begin{equation}\label{e1-1}
\lim_{|x| \to \infty}V(x)=\infty
\end{equation}
was commonly used in \cite{KS,KK, KL} to derive the compactness
of $(T_t^V)_{t \ge 0}$, e.g. \cite[Assumption 2.4]{KL}. However, as shown by Proposition \ref{p1-2}, assumption
{\bf (A2)}, which is much weaker than (\ref{e1-1}), is sufficient to ensure the compactness of
$(T_t^V)_{t \ge 0}$. Therefore, a natural question is whether one can give some sufficient conditions
for the intrinsic ultracontractivity of $(T_t^V)_{t \ge 0}$ without the restrictive condition
(\ref{e1-1}).

In this paper, we will make use of super Poincar\'e inequalities
with respect to infinite measure developed in \cite{Wang02} and
functional inequalities for non-local Dirichlet forms recently
studied in \cite{WW,WJ13,CW14} to deal with the questions mentioned above. We aim to present some sharp
conditions on the potential function $V$ such that the associated
Feynman-Kac semigroup $(T_t^V)_{t\ge0}$ is intrinsically
ultracontractive, and also derive explicit two-sided estimates for
the ground state $\phi_1$.  Our method is different from that of \cite{KS,KK,KL}, and deals with the intrinsic ultracontractivity of $(T_t^V)_{t\ge0}$
for non-local Dirichlet forms in more general situations. The following points indicate the novelties of
our paper.

\begin{itemize}
\item[(i)] We can deal with the example $J(x,y)$ mentioned
in \eqref{jjj}, which essentially means that small jumps play the dominant
roles for the behavior of the associated process. On the other hand, by \eqref{e3-1}, we also consider the case that the density of the small jumps
can enjoy the variable order property.

\item[(ii)] For a large class of
potential functions $V$ which do not satisfy the growth condition (\ref{e1-1}) or
the regularity condition (\ref{e1-1a}) below, we can still obtain some
sufficient conditions for the intrinsic ultracontractivity of
$(T_t^V)_{t \ge 0}$, which to the
best of our knowledge do not appear in the literature.

\item[(iii)]
Our method here efficiently
applies to Hunt process generated by non-local Dirichlet forms. In
particular, the associated process does not like L\'{e}vy process,
and it is usually not space-homogeneous.
\end{itemize}

\bigskip

Now, we will present main results of our paper, which will be split into two subsections.

\subsubsection{{\bf The case that $\lim\limits_{|x|\to\infty} V(x)=\infty$.}} The following statement is a consequence of more general Theorem  \ref{t3-1}
below.

\begin{theorem}\label{thm2}
Suppose that \eqref{e3-1}, \eqref{e3-2}, {\bf (A1)} and {\bf (A2)}
hold, and that there exist positive constants $c_i$ $(i=3,4)$,
$\theta_i$ $(i=1,3)$ and constants $\theta_i$ $(i=2,4)$ such that
for every $x \in \R^d$ with $|x|$ large enough,
\begin{equation}\label{vv} c_3|x|^{\theta_1}\log^{\theta_2}(1+|x|) \le V(x)\le
c_4|x|^{\theta_3}\log^{\theta_4}(1+|x|).\end{equation} If
$\theta_1=1$ and $\theta_2>2$ or if $\theta_1>1$, then $(T_t^V)_{t \ge 0}$ is
intrinsically ultracontractive, and for any $\varepsilon>0$, there
exists a constant $c_5=c_5(\varepsilon)>0$ such that for all
$x\in\R^d$,
$$ c_5\exp\Big(- \frac{(1+\varepsilon)\theta_3}{\kappa}|x|\log(1+|x|)\Big)\le
\phi_1(x).$$

Additionally, if {\bf (A3)} also holds and
 $$J(x,y)=0,\quad x,y\in\R^d\textrm{ with } |x-y|>\kappa,$$ then for any $\varepsilon>0$, there exists a constant
$c_6=c_6(\varepsilon)>0$ such that for all $x\in\R^d$,
$$
 \phi_1(x)\le  c_6\exp\Big(-\frac{(1-\varepsilon)\theta_1}{\kappa}
 |x|\log(1+|x|)\Big).
$$
\end{theorem}

To show that Theorem \ref{thm2} is sharp, we have the following
example, which, as mentioned above, can not be studied by the method
used in \cite{KS, KK, KL}.

\begin{example}\label{ex2}\it
Suppose that assumptions {\bf (A1)}, {\bf (A2)} and {\bf (A3)} hold,
and
$$J(x,y)\asymp|x-y|^{-d-\alpha}\I_{\{|x-y|\le
1\}}+e^{-|x-y|^\gamma}\I_{\{|x-y|> 1\}},$$ where $\alpha \in (0,2)$
and $\gamma\in(1,\infty]$. If $V(x)=|x|^{\theta}$ for some constant
$\theta>0$, then the semigroup $(T_t^V)_{t \ge 0}$ is intrinsically
ultracontractive if and only if $\theta>1$. When $\theta>1$, we have the following explicit two-sided estimates for the
ground state $\phi_1$.
\begin{enumerate}
\item[(1)] If $\gamma=\infty$, i.e. the associated Hunt process $(X_t)_{t\ge0}$ is with finite range jumps,
then for any $\varepsilon\in(0,1)$, there exist $c_i=c_i(\varepsilon,\theta)$ $(i=1,2)$ such that for all $x\in\R^d$,
\begin{equation}\label{ex2-1}
\begin{split}c_1\exp\Big(-(1+\varepsilon)\theta|x|\log(1+|x|)\Big)& \le
\phi_1(x)\\
&\le
c_2\exp\Big(-(1-\varepsilon)\theta|x|\log(1+|x|)\Big).\end{split}
\end{equation}
\item[(2)] If $1<\gamma<\infty$, then there exist positive constants $c_i:=c_i(\gamma)$ $(i=4, 6)$ independent of $\theta$ such that for all $x\in\R^d$,
\begin{equation}\label{ex2-2}
\begin{split}c_3\exp\Big(-c_4 \theta^{\frac{\gamma-1}{\gamma}}|x|\log^{\frac{\gamma-1}{\gamma}}(1+|x|)\Big) &\le
\phi_1(x)\\
&\le
c_5\exp\Big(-c_6 \theta^{\frac{\gamma-1}{\gamma}}|x|\log^{\frac{\gamma-1}{\gamma}}(1+|x|)\Big)\end{split}
\end{equation}
holds for some positive constants $c_3=c_3(\theta,\gamma)$ and $c_5(\theta,\gamma)$.
\end{enumerate}
\end{example}

We make some comments on Theorem \ref{thm2} and Example \ref{ex2}.
\begin{remark}
(1) Compared with \cite{KS,KK, KL}, to ensure the intrinsic ultracontractivity of $(T_t^V)_{t \ge 0}$ Theorem \ref{thm2}
gets rid of the following restrictive condition on potential function $V$:
\begin{equation}\label{e1-1a}
\sup_{z\in B(x,1)}V(z)\le C V(x),\quad |x|\ge1,
\end{equation}
 see e.g.\ \cite[Assumption 2.5 and Corollary 2.3(1)]{KL}.
 Intuitively, regularity condition (\ref{e1-1a}) means that the rate for the oscillation of
 $V$ is mild. However,  according to \eqref{vv}, we know from Theorem \ref{thm2} that $(T_t^V)_{t \ge 0}$ still may be intrinsically
 ultracontractive
without such regular condition on $V$. The reader can refer to Proposition \ref{pro} below for more general conditions on $V$.
 Roughly speaking, the upper bound for $V$ in \eqref{vv} is used to control the lower bound for the ground state $\phi_1$, while the lower bound for $V$ is needed to establish the upper bound estimate for $\phi_1$, and also the intrinsic (local) super Poincar\'{e} inequality for Dirichlet form $(D,\D(D))$.

(2) In L\'{e}vy case, if $V(x)=|x|^{\theta}$ for some $\theta>0$,
the conclusion of Example \ref{ex2} says that $(T_t^V)_{t \ge 0}$ is intrinsically ultracontractive if and
only if $\theta>1$. Such condition on $V$ is the
same as that in case of $\gamma=1$, which is associated with the Feynman-Kac
semigroup for relativistic $\alpha$-stable processes, see
\cite[Theorem 1.6 and the remark below]{KS} for more details.
However,
the case $\gamma\in(1,\infty]$ does not fit the framework of
\cite{KS,KK,KL}, and it is essentially different from the case $\gamma \in
(0,1)$. Indeed, let $\rho$ be the density function of the L\'{e}vy
measure. According to \cite[Assumption 2.1]{KL}, the function $\rho$
is required to satisfy that
\begin{itemize}
\item[(i)] There exists a constant $C_1>0$ such that for every $1\le |y|\le
|x|$,
$$\rho(y)\le C_1\rho(x).$$
\item[(ii)] There exists a constant $C_2>0$ such that for all $x$,
$y\in\R^d$ with $|x-y|>1$,
$$\int_{\{|z-x|\ge1, |z-y|\ge 1\}} \rho(x-z)\rho(z-y)\,dz\le
C_2\rho(x-y).$$
\end{itemize}
By \cite[Example 4.1 (3)]{KL}, the assumptions (i) and
(ii) are only satisfied when $\gamma\in(0,1]$. On the other hand,
the difference between $\gamma>1$ and $0<\gamma\le 1$ is also
indicated by \cite[Theorem 1.2 (1) and (2)]{CKK1}, where explicit
global heat kernel estimates of the associated process (depending on
the parameter $\gamma$) are presented.


(3) In Example \ref{ex2} (1), i.e.\ $\gamma=\infty$, the symmetric Hunt process associated with density function $J$ above is the truncated symmetric $\alpha$-stable-like process, e.g.\ see \cite{CKK2}. On the other hand,
if the Hunt process is a Brownian motion and $V(x)=|x|^\theta$ for
some $\theta>0$, then, according to \cite[Theorem 6.1]{DS} (at least
in one dimension case), we know that the associated Feynman-Kac
semigroup is intrinsically ultracontractive if and only if
$\theta>2$.  This, along with Example \ref{ex2} (1), indicates the
difference of the intrinsic ultracontractivity for Feynman-Kac
semigroups between L\'evy process (symmetric jump processes) with finite range jumps and
Brownian motion.
\end{remark}

\ \

\subsubsection{{\bf The case that $\lim\limits_{|x|\to\infty} V(x)\neq\infty$.}}
The following theorem gives us sufficient conditions on the intrinsic ultracontractivity
of $(T_t^V)_{t \ge 0}$ for a class of irregular potential functions $V$ such that $\lim\limits_{|x|\to\infty} V(x)\neq\infty$.
Denote by $|A|$
Lebesgue measure of a Borel set $A\subseteq \R^d$.
\begin{theorem}\label{thm3}
Suppose that \eqref{e3-1}, \eqref{e3-2}, assumptions {\bf (A1)} and {\bf (A2)} hold,  and that there
exists a unbounded subset $A\subseteq \R^d$ such that the following conditions are satisfied.
\begin{enumerate}
\item [(1)] $|A|<\infty$ and $$A \cap \{x \in \R^d: |x|\ge R\}\neq\emptyset, \quad\forall\ R>0.$$

\item  [(2)] There
exist positive constants $c_i$ $(i=3,4)$, $\theta_i$ $(i=1,2)$ with $\theta_1>2$ and
constant $\theta_3\in\R$ such that for all $x \in \R^d$ with $|x|$ large enough,
$$ V(x)=1,\quad  x \in A$$ and
$$c_3|x|\log^{\theta_1}(1+|x|) \le V(x)\le
c_4|x|^{\theta_2}\log^{\theta_3}(1+|x|),\quad  x \notin A.$$

\item [(3)] There exist positive constants $c_i$ $(i=5,6)$ and $\eta_i$ $(i=1,2)$ such that
for every $R>2$,
$$|\{x \in \R^d: x \in A, |x|\ge R\}|\le c_5 \exp(-c_6R^{\eta_1}\log^{\eta_2}R).$$
\end{enumerate}
Then, we have
\begin{enumerate}
\item [(i)] If $\eta_1=1$ and $\eta_2>1$, then the associated
Feynman-Kac
semigroup $(T_t^V)_{t \ge 0}$  is intrinsically ultracontractive.

\item [(ii)] If $\eta_1=\eta_2=1$, then there exists a constant
$c_0>0$ such that for any $c_6>c_0$, the associated
semigroup $(T_t^V)_{t \ge 0}$  is intrinsically ultracontractive.
\end{enumerate}

\ \

Suppose moreover $d>\alpha_1$, and replace $(3)$ by the following weaker condition
\begin{itemize}
\item[(4)] There exist positive constants $c_7$ and $\eta_3$ such that
for every $R>2$,
\begin{equation}\label{power} |\{x \in \R^d: x \in A, |x|\ge R\}|\le \frac{c_7}{ R^{d/\alpha_1}\log^{\eta_3}R} .\end{equation}
\end{itemize}
Then, if $d>\alpha_1$, $(1)$, $(2)$ and $(4)$ hold with $\eta_3>{2d}/{\alpha_1}$, then the associated
semigroup $(T_t^V)_{t \ge 0}$ is intrinsically ultracontractive.
\end{theorem}

{ The following example shows that one can not
replace the decay rate $d/\alpha_1$ in \eqref{power} by $d/\alpha_1-\varepsilon$ with any
$\varepsilon>0$.

\begin{example}\label{ex3}\it Consider the truncated symmetric $\alpha$-stable process on $\R^d$  with
some $0<\alpha<2$, i.e.  $$J(x,y)=|x-y|^{-d-\alpha},\quad 0<|x-y|\le 1$$
and $$J(x,y)=0,\quad |x-y|>1.$$

For any $\varepsilon\in(0,1)$, let $A=\bigcup_{n=1}^{\infty}B(x_n,r_n)$ be such that $x_n \in \R^d$ with $|x_n|=n^{k_0}$ and
$r_n=n^{-\frac{k_0}{\alpha}+\frac{1}{d}}$ for $n \ge 1$, where $k_0>\frac{2}{\varepsilon}$.
Suppose that
\begin{equation*}
V(x)=\begin{cases}
 \,\,1,\ \ \ \ \ \text{if}\ \ x \in A,\\
 |x|^\theta,\ \ \  \text{if}\ \ x \notin A
\end{cases}
\end{equation*} with some constant
$\theta>1$. Then $(T_t^V)_{t \ge 0}$ is not intrinsically ultracontractive.

However, there is a constant $c_0>0$ such that
\begin{equation}\label{ex3-0}
|\{x \in \R^d: x \in A, |x|\ge R\}|\le \frac{c_0}{ R^{\frac{d}{\alpha}-\varepsilon}},\quad R>2.
\end{equation}

\end{example}}

\ \

The remainder of this paper is arranged as follows. In Section
\ref{section2}, we will present sufficient conditions for the intrinsic
ultracontractivity of Feynman-Kac semigroup in terms of
intrinsic super Poincar\'e inequality, see Theorem \ref{p2-1}. These conditions are
interesting of themselves, and they work for general framework including local Dirichlet forms and non-local Dirichlet forms.
Section \ref{section3} is devoted to applying Theorem \ref{p2-1} to
yield general results about the
intrinsic ultracontractivity of the Feynman-Kac semigroups for non-local Dirichlet forms.
We use the probabilistic method and the iterated
approach to derive an explicit lower bound estimate for ground state
of the semigruoup $(T_t^V)_{t\ge0}$, e.g.\ Proposition \ref{p3-1}.
The intrinsic local super Poincar\'e inequality for the Dirichlet
form $(D^V,\D(D^V))$ is established in Proposition \ref{l3-4}.
Proofs of all the statements in Section \ref{section1} are presented in Section \ref{section4}, and proofs of Propositions \ref{p1-1} and \ref{p1-2} are given in Appendix.

\ \

\noindent {\bf Notation}\,\, Throughout this paper, let $d\ge 1$. By
$|x|$ we denote the Euclidean norm of $x\in\R^d$, and by $|A|$ the
Lebesgue measure of a Borel set $A$. Denote by $B(x,r)$ the ball
with center $x \in \R^d$ and radius $r>0$. For any $A$,
$B\subset\R^d$, let $dist (A,B)=\inf\{|x-y|: x\in A, y\in B\}$. We
will write $C=C(\kappa,\delta,\varepsilon,\lambda,\ldots)$ to
indicate the dependence of the constant $C$ on parameters. The
constants may change their values from one line to the next, even on
the same line in the same formula. Let
$B_b(\R^d)$ be the set of bounded measurable functions on $\R^d$. For any measurable functions $f$,
$g$ and any $\sigma$-finite measure $\mu$ on $\R^d$, we set $\langle
f, g\rangle_{L^2(\R^d; \mu)}:=\int f(x)g(x)\,\mu(dx)$, and for any $p\in[1,\infty)$, $\|f\|_{L^p(\R^d;\mu)}:=\big(\int |f(x)|^p\,\mu(dx)\big)^{1/p}$.
Denote by $\|f\|_{\infty}$ the $L^{\infty}(\R^d;dx)$-norm for any
bounded function $f$. For any increasing function $f$ on
$(0,\infty)$,  $f^{-1}(r):=\inf\{s>0: f(s)\ge r\}$ is its right
inverse.

\section{Intrinsic Ultracontractivity for General Dirichlet Forms}\label{section2}
The aim of this section is to present sufficient conditions for
intrinsic ultracontractivity of Feynman-Kac semigroup associated with general symmetric Dirichlet forms (including local Dirichlet forms).
Since we believe that the result below is interesting of itself and has wide applications, for sake of self-containedness we first introduce some necessary notations even if they are repeated by previous section.

 \ \

Let $(D,\mathscr{D}(D))$ be a regular symmetric Dirichlet form (not necessarily non-local) on $L^2(\R^d,dx)$ with core $C_c^2(\R^d)$, and let $V$ be a locally bounded non-negative measurable function on $\R^d$. Consider the following regular Dirichlet form with killing on $L^2(\R^d,dx)$:
$$ D^V(f,f) =D(f,f)+\int f^2(x)V(x)\,dx, \quad \mathscr{D}(D^V)= \overline{C_c^2(\R^d)}^{{D_1^V}},$$
where $${D_1^V}(f,f):={D^V(f,f)+\|f\|_{L^2(\R^d;dx)}^2}.$$ Denote by $(T_t^V)_{t\ge 0}$ the associated (Feynman-Kac) semigroup on $L^2(\R^d,dx)$.
To consider the intrinsic ultracontractivity of Feynman-Kac semigroup $(T_t^V)_{t\ge 0}$, we assume that
\begin{itemize}
\item[(A)] The Feynman-Kac semigroup $(T_t^V)_{t\ge 0}$ is compact on $L^2(\R^d,dx)$, and its ground state
$\phi_1$  corresponding to the first eigenvalue $\lambda_1>0$ is bounded, continuous and strictly positive.
\item[(B)] The potential function $V$ satisfies
\begin{itemize}
\item[{\bf (A2)}]  \emph{For every $r>0$, $$|\{x\in\R^d: V(x)\le r\}|<\infty.$$}

\item[{\bf (A4)}]
\emph{There exists a constant $K>0$ such that
$$\lim_{R \to\infty} \Phi(R)=\infty,$$
where
$$\Phi(R)=\Phi_K(R):=\inf_{|x|\ge R, V(x)>K} V(x),\quad \ R>0.$$}

\end{itemize}
\end{itemize}
According
to assumption {\bf (A2)}, $\big|\{x \in \R^d: V(x)\le
K\}\big|<\infty$. Therefore, assumption {\bf (A4)} means
that the potential function $V$ tends to infinity as $|x| \to
\infty$ on the complement of a set (maybe unbounded) with finite
Lebesgue measure. Obviously both {\bf (A2)} and {\bf (A4)} hold true when
\begin{equation}\label{e2-1}
\lim_{|x| \rightarrow \infty}V(x)=\infty.
\end{equation}

For the constant $K$ in assumption {\bf (A4)}, let
$$
\Theta(R)=\Theta_K(R):=\big|\{x \in \R^d: |x|\ge R, V(x)\le
K\}\big|,\quad R>0.
$$
On the other hand, due to the fact $\big|\{x \in \R^d: V(x)\le K\}\big|<\infty$, it is easy to
see that
\begin{equation*}
\lim_{R \to \infty}\Theta(R)=0.
\end{equation*} In particular, if \eqref{e2-1} holds, then for any constant $K>0$,
$\Theta(R)=0$ when $R>0$ is large enough.

\ \

Now, we sate the main result in this section.

\begin{theorem}\label{p2-1}
Let $(T_t^V)_{t\ge 0}$ be a compact Feynman-Kac semigroup on $L^2(\R^d,dx)$, and $V$ be a locally bounded non-negative measurable function such that Assumptions $(A)$ and $(B)$ are satisfied. Suppose that there exists a bounded
measurable function $\varphi \in B_b(\R^d)$ such that the following
conditions are satisfied.
\begin{enumerate}
\item[(1)] There is  a constant $r_0>0$ such that for every $r\ge r_0$, the following local intrinsic super Poincar\'e inequality
\begin{equation*}
\begin{split}
\int_{B(0,r)}f^2(x)\,dx\le & sD^V(f,f)\\
&+\alpha(r,s)\Big(\int |f|(x)\varphi(x)\,dx\Big)^2, \quad s>0,\
f \in \D(D^V)
\end{split}
\end{equation*}
holds for some positive measurable function $\alpha$.

\item[(2)] Let
$\phi_1$ be the ground state for the semigroup $(T_t^V)_{t\ge0}$.
It holds for some constant $C_0>0$ that
\begin{equation*}
\varphi(x)\le C_0\phi_1(x),\quad x\in\R^d.
\end{equation*}
\end{enumerate}
Then, the following intrinsic mixed type super Poincar\'e inequality
\begin{equation}\label{mixed}
\begin{split}
\int f^2(x)\,dx\le s D^V(f,f)&+\beta(s\wedge s_0) \left(\int|f|(x)\phi_1(x) \,dx\right)^2\\
&
+\gamma(s\wedge s_0)^{(p-2)/p}\|f\|_{L^p(\R^d;dx)}^2
\end{split}
\end{equation}
holds for all $s>0$, $f\in \D(D^V)$ and $p\in(2,\infty]$ with the rate functions
\begin{equation}\label{p2-1-5}
\beta(s):=C_0^2\alpha\left( \Phi^{-1}\left(\frac{2}{s}\right),
\frac{s}{2}\right),\quad \gamma(s):=
\Theta\left(\Phi^{-1}\left(\frac{2}{s}\right)\right)
\end{equation}
and constant  $s_0=\frac{2}{\Phi(r_0)}$. Here, we use the convention that
$(p-2)/p=1$ when $p=\infty$.

Moreover, we have \begin{enumerate}
\item[(i)] If \eqref{e2-1} holds, then the following super Poincar\'e inequality
\begin{equation}\label{p2-1-6}
\begin{split}
\int f^2(x)\,dx\le& s D^V(f,f)\\
&+\beta(s \wedge r_1) \left(\int |f|(x)\phi_1 (x)\,dx\right)^2
,\quad s>0, f\in \D(D^V)
\end{split}
\end{equation}
holds for some constant $r_1>0$. Consequently, if
\begin{equation*}\label{t2-1-1a}
\int_t^{\infty}\frac{\beta^{-1}(s)}{s}\,ds<\infty,\quad t>\inf
\beta,
\end{equation*}
then the semigroup $(T^V_t)_{t \ge 0}$ is intrinsically
ultracontractive.

\item[(ii)] If for some $p>2$ there is a constant $c_0>0$ such that the following Sobolev inequality holds true
\begin{equation}\label{p2-2-1}
\|f\|_{L^p(\R^d;dx)}^2 \le c_0\bigg[ D^V(f,f)+\|f\|_{L^2(\R^d;dx)}^2\bigg],\quad f \in C_c^{\infty}(\R^d),
\end{equation}
then the super Poincar\'e inequality \eqref{p2-1-6} holds with
the rate function $\beta$ and the constant $r_1$ replaced by
\begin{equation}\label{p2-2-3}
{\hat \beta(s)}:=2C_0^2\alpha\left( \Psi^{-1}\left(\frac{s}{4}\right),
\frac{s}{4}\right)
\end{equation} and some constant $r_2>0$ respectively,
where $$\Psi(R):=\frac{1}{\Phi(R)}+c_0\Theta(R)^{\frac{p-2}{p}},\quad R>1.$$ {Consequently, if}
\begin{equation*}\label{t2-1-1a}
\int_t^{\infty}\frac{{\hat \beta^{-1}(s)}}{s}\,ds<\infty,\quad t>\inf
\hat \beta,
\end{equation*}
then the semigroup $(T^V_t)_{t \ge 0}$ is intrinsically
ultracontractive.

\item[(iii)] If there exists a constant $\delta>1$ such that
\begin{equation}\label{p2-1-7}
\sum_{n=1}^{\infty}\gamma(s_n)\delta^n<\infty,
\end{equation}
where $s_n:=\beta^{-1}(\frac{c_1 \delta^n}{2})$ with
$c_1:=\|\phi_1\|_{\infty}^2$, then the following super Poincar\'e
inequality holds
\begin{equation}\label{p2-1-8}
\begin{split}
\int f^2(x)\,dx\le & s D^V(f,f)\\
&+\tilde \beta(s\wedge r_2) \left(
\int|f|(x)\phi_1(x)\,dx\right)^2 ,\quad s>0,\ f\in \D(D^V),
\end{split}
\end{equation}
where $r_2$ is a positive constant, $$ \tilde
\beta(s):=2\beta\left(\gamma^{-1}\left(\frac{1}{4\delta^{n_0(s)+1}}\right)\right)$$
and
\begin{equation}\label{p2-1-8a}
\begin{split}
n_0(s):=\inf\Bigg\{N \ge&\left(\log_\delta\Big(\frac{2\beta(s_0)}{c_1}\Big)\right)\vee\left(-\log_\delta\big(4\delta\gamma(s_0)\big)\right):\\
&\frac{4\delta(\sqrt{\delta}+1)s_{N}}{\sqrt{\delta}-1}
+2\gamma^{-1}\left(\frac{1}{4\delta^{N+1}}\right)\le s\Bigg\}.
\end{split}
\end{equation}
Consequently, if
\begin{equation*}\label{t2-1-2}
\int_t^{\infty}\frac{\tilde \beta^{-1}(s)}{s}\,ds<\infty,\quad
t>\inf \tilde \beta,
\end{equation*}
then the semigroup $(T^V_t)_{t \ge 0}$ is intrinsically
ultracontractive.
\end{enumerate}
\end{theorem}

\begin{proof} Throughout the proof, we denote by $\mu$ the Lebesgue measure on $\R^d$.
For the constant $K$ in assumption {\bf (A4)},
let $A_1:=\{x \in \R^d: V(x)> K\}$ and  $A_2:=\R^d\setminus A_1$.
For any $f \in \D(D^V)$, $R\ge r_0$ and $p\in(2,\infty]$, it holds that
\begin{equation}\label{ddd-222}
\begin{split}
\int_{B(0,R)^c}f^2(x)\,\mu(dx)&= \int_{B(0,R)^c\bigcap
A_1}f^2(x)\,\mu(dx)
+\int_{B(0,R)^c\bigcap A_2}f^2(x)\,\mu(dx)\\
&\le \frac{1}{\Phi(R)}
\int_{B(0,R)^c \bigcap A_1}f^2(x)V(x)\,\mu(dx)\\
&\quad+\mu(B(0,R)^c\cap A_2)^{(p-2)/p}\|f\|^2_{L^p(\R^d;dx)}\\
& \le \frac{1}{\Phi(R)}D^V(f,f)+\Theta(R)^{(p-2)/p}\|f\|_{L^p(\R^d;dx)}^2,
\end{split}
\end{equation} where in the first inequality we have used the H\"{o}lder inequality when $p\in(2,\infty)$.
This, along with conditions (1) and (2), gives us that for any $R$, $\tilde s>0$ and $f\in \D(D^V)$,
\begin{equation*}
\begin{split}
\mu(f^2)\le \left(\frac{1}{\Phi(R)}+\tilde s\right) D^V(f,f)+
C_0^2\alpha(R,\tilde s)\mu\big(\phi_1 |f|\big)^2+
\Theta(R)^{(p-2)/p}\|f\|_{L^p(\R^d;dx)}^2.
\end{split}
\end{equation*}
For any $0<s\le s_0:=\frac{2}{\Phi(r_0)}$, taking $R=\Phi^{-1}\left(\frac{2}{s}\right)$
and $\tilde s=\frac{s}{2}$ in the inequality above, we can get the
required mixed type super Poincar\'{e} inequality (\ref{mixed}) for all $s\in(0,
s_0]$. Hence, the proof of the first
assertion is completed by choosing $\beta(s)=\beta(s_0)$ and
$\gamma(s)=\gamma(s_0)$ for all $s\ge s_0$.

(i) We take $p=\infty$ in \eqref{ddd-222}. Suppose that (\ref{e2-1}) holds. Then $\Theta(R)=0$ for $R>0$
large enough. This immediately yields the true super Poincar\'{e}
inequality (\ref{p2-1-6}) with some constant $r_1>0$.

Let $(L^V, \mathscr{D}(L^V))$ be the generator associated with $(T^V_t)_{t\ge0}$, and
$(\tilde T_t^V)_{t\ge0}$ be the strongly continuous semigroup defined by (\ref{e1}).
Due to the fact that $L_V \phi_1=-\lambda_1 \phi_1$, the (regular)
Dirichlet form $(D_{\phi_1},\mathscr{D}(D_{\phi_1}))$ associated
with $(\tilde T_t^V)_{t \ge 0}$ enjoys the properties that,
$C_c^2(\R^d)$ is a core for  $(D_{\phi_1},\mathscr{D}(D_{\phi_1}))$,
and for any $f\in C_c^2(\R^d)$,
\begin{equation}\label{t2-1-1}
\begin{split}
D_{\phi_1}(f,f)=D^V(f\phi_1,f\phi_1)-\lambda_1 \int_{\R^d}
f^2(x)\phi_1^2(x)\,dx.
\end{split}
\end{equation}
Let $\mu_{\phi_1}(dx)=\phi_1^2(x)\,dx$. Combining (\ref{t2-1-1})
with (\ref{p2-1-6}) gives us the following intrinsic super Poincar\'e inequality
\begin{equation*}\label{t2-1-2a}
\begin{split}
\mu_{\phi_1}(f^2)&\le  s D^V(f\phi_1,f\phi_1)+
\beta(s\wedge r_1) \Big(\int |f|(x)\phi^2_1(x)\,dx\Big)^2\\
&\le s \Big(D_{\phi_1}(f,f)+\lambda_1\mu_{\phi_1}(f^2)\Big)+\beta(s\wedge r_1) \mu_{\phi_1}^2(|f|),
\quad s>0,
\end{split}
\end{equation*}
where the rate function $\beta(s)$ is given by (\ref{p2-1-5}). In
particular, for any $s\in(0,1/(2\lambda_1))$,
\begin{equation*}
\mu_{\phi_1}(f^2)\le 2sD_{\phi_1}(f,f)+2\beta({s}\wedge r_1),\quad f\in C_c^2(\R^d),
\end{equation*}which implies that
\begin{equation*}
\mu_{\phi_1}(f^2)\le sD_{\phi_1}(f,f)+2\beta\Big(\frac{s}{2}\wedge r_1\wedge \frac{1}{\lambda_1}\Big),\quad f\in C_c^2(\R^d), s>0.
\end{equation*}

Therefore, the desired assertion in (i) for the ultracontractivity of the (Markovian)
semigroup $(\tilde T_t^V)_{t \ge 0}$ (or, equivalently, the intrinsic ultracontractivity of the semigroup $(T_t^V)_{t \ge 0}$)
follows from \cite[Theorem
3.3.13]{WBook} or \cite[Theorem 3.1]{Wang00}.

(ii) Suppose \eqref{p2-2-1} holds true. According to \eqref{ddd-222} and \eqref{p2-2-1}, for any $f \in C_c^\infty(\R^d)$ and $R\ge r_0$,
\begin{equation*}
\begin{split}
\int_{B(0,R)^c}f^2(x)\,\mu(dx)&\le  \left( \frac{1}{\Phi(R)}+c_0\Theta(R)^{({p-2})/{p}}\right)
D^V(f,f)+c_0\Theta(R)^{({p-2})/{p}}\mu(f^2),
\end{split}
\end{equation*}
Combining this with conditions (1) and (2) in Theorem \ref{p2-1},
we have that for any $R\ge R_1\vee r_0$ with $c_0\Theta(R_1)^{({p-2})/{p}}\le{1}/{2}$, each $\tilde s>0$ and $f\in C_c^\infty(\R^d)$,
\begin{equation*}
\begin{split}
\mu(f^2) &\le 2\left(\frac{1}{\Phi(R)}+c_0\Theta(R)^{{(p-2)/}{p}}+\tilde s\right) D^V(f,f)+
2C_0^2\alpha(R,\tilde s)\mu\big(\phi_1 |f|\big)^2\\
&=2\left(\Psi(R)+\tilde s\right) D^V(f,f)+2C_0^2\alpha(R,\tilde s)\mu\big(\phi_1 |f|\big)^2.
\end{split}
\end{equation*}
Taking $R=\Psi^{-1}\left(\frac{s}{4}\right)$
and $\tilde s=\frac{s}{4}$ in the inequality above for $s>0$ small enough, we can get the super Poincar\'e inequality \eqref{p2-1-6} with the desired rate function $\hat{\beta}$ given by \eqref{p2-2-3}, also thanks to the fact that $C_c^\infty(\R^d)$ is dense in $\mathscr{D}(D^V).$
Having this at hand, we can arrive at the final assertion for (ii) by the same argument as that in the proof of part (i).

(iii) Now we take $p=\infty$ in \eqref{ddd-222}, and assume that (\ref{p2-1-7}) holds. Note that, for every $f
\in \D(D^V)$, $|f|\in \D(D^V)$ and $D^V(|f|,|f|)\le D^V(f,f)$, and
so it suffices to prove that (\ref{p2-1-8}) holds for any $f \in
\D(D^V)$ with $f \ge 0$.

Given any $f \in \D(D^V)$ with $f \ge 0$ and $\mu(f^2)=1$, for any $\delta>1$, we have
\begin{equation}\label{aaa}
\begin{split}
\mu(f^2)&=\int_0^{\infty}\mu(f^2>t)\,dt\\
&=\int_0^{\delta^{n_0+1}}\mu(f^2>t)\,dt
+\sum_{n=n_0+1}^{\infty}\int_{\delta^n}^{\delta^{n+1}}\mu(f^2>t)\,dt\\
&\le \mu\big((f\wedge \delta^{\frac{n_0+1}{2}})^2\big)+
\sum_{n=n_0+1}^{\infty}(\delta^{n+1}-\delta^{n})\mu(f^2>\delta^n)\\
&=:J_{n_0}+\sum_{n=n_0+1}^{\infty}I_n,
\end{split}
\end{equation}
where $n_0$ is an integer to be determined later.

Next, we define
\begin{equation*}
f_n:=\big(f-\delta^{\frac{n}{2}}\big)^+\wedge
\big(\delta^{\frac{n+1}{2}}-\delta^{\frac{n}{2}}\big),\quad n\ge0.
\end{equation*}
Noticing that for any $n\ge0$, $$f_n \ge
\big(\delta^{\frac{n+1}{2}}-\delta^{\frac{n}{2}}\big)\I_{\{f^2>\delta^{n+1}\}},$$
we get \begin{equation}\label{p2-1-11}
\begin{split}
I_n&=(\delta^{n+1}-\delta^{n})\mu(f^2>\delta^n)\\
&\le
\frac{(\delta^{n+1}-\delta^{n})\mu(f_{n-1}^2)}{\big(\delta^{\frac{n}{2}}-\delta^{\frac{n-1}{2}}\big)^2}\\
&=\frac{\delta(\sqrt{\delta}+1)}{\sqrt{\delta}-1}\mu(f_{n-1}^2),\quad
n\ge0.
\end{split}
\end{equation}
According to \eqref{mixed} and the fact that $f_n \in \D(D^V)$, for all
$n\ge0$ and $0<s\le s_0$,
\begin{equation}\label{p2-1-10}
\mu(f_n^2)\le sD^V(f_n,f_n)+\beta(s)\mu(\phi_1
f_n)^2+\gamma(s)\delta^n(\sqrt{\delta}-1)^2.
\end{equation}
Due to the Cauchy-Schwarz inequality, the Chebyshev inequality and the fact that $\mu(f^2)=1$,
\begin{equation*}
\begin{split}
\mu(\phi_1 f_n)^2&=\mu(\phi_1 f_n \I_{\{f^2>\delta^n\}})^2\le
\mu(\phi_1^2\I_{\{f^2>\delta^n\}})\mu(f_n^2)\le
c_1\delta^{-n}\mu(f_n^2).
\end{split}
\end{equation*}
 Then, taking $s=s_n:=\beta^{-1}\left(\frac{c_1\delta^n}{2}\right)$ with $n\ge\log_\delta\Big(\frac{2\beta(s_0)}{c_1}\Big)$
in \eqref{p2-1-10}, we obtain
\begin{equation*}
\mu(f_n^2)\le s_n
D^V(f_n,f_n)+\frac{1}{2}\mu(f_n^2)+(\sqrt{\delta}-1)^2\gamma(s_n)\delta^n,
\end{equation*}
which implies that
\begin{equation*}\label{p2-1-10a}
\mu(f_n^2)\le 2s_n
D^V(f_n,f_n)+2(\sqrt{\delta}-1)^2\gamma(s_n)\delta^n.
\end{equation*}
Since (\ref{p2-1-7}) holds true, there exists an integer $$n_0\ge\left(\log_\delta\Big(\frac{2\beta(s_0)}{c_1}\Big)\right)\vee\left(-\log_\delta\big(4\delta\gamma(s_0)\big)\right)$$
such that
$${\delta({\delta}-1)}\sum_{i=n_0}^{\infty}\gamma(s_i)\delta^i\le \frac{1}{8}.$$

Furthermore, it is easy to see that
\begin{equation*}
\sum_{n=0}^{\infty}|f_n(x)-f_n(y)|\le |f(x)-f(y)|,\ \ \sum_{n=0}^{\infty}|f_n(x)|\le |f(y)|,
\end{equation*}
so, according to \cite[Lemma 3.3.2]{WBook},
$$
\sum_{n=0}^{\infty}D^V(f_n,f_n)\le D^V(f,f).
$$

Combining all the estimates above with (\ref{p2-1-11}), and noting that $s_n$ is
non-increasing with respect to $n$, we arrive at
\begin{equation}\label{p2-1-11a}
\sum_{n=n_0+1}^{\infty}I_n\le\frac{2\delta(\sqrt{\delta}+1)s_{n_0}}{\sqrt{\delta}-1}
D^V(f,f)+\frac{1}{4}.
\end{equation}

On the other hand, applying $f\wedge \delta^{\frac{n_0+1}{2}}$ into
(\ref{p2-1-6}), we have
\begin{equation*}
\begin{split}
J_{n_0}&=\mu\big((f \wedge \delta^{\frac{n_0+1}{2}})^2\big)\\
& \le s D^V(f \wedge \delta^{\frac{n_0+1}{2}},f \wedge
\delta^{\frac{n_0+1}{2}})
+\beta(s)\mu(\phi_1 f)^2+\gamma(s)\delta^{n_0+1}\\
&\le sD^V(f,f)+\beta(s)\mu(\phi_1 f)^2+\gamma(s)\delta^{n_0+1},\quad 0<s\le s_0,
\end{split}
\end{equation*}
where the second inequality also follows from \cite[Lemma
3.3.2]{WBook}. Hence, noticing that
 $n_0\ge -\log_\delta\big(4\delta\gamma(s_0)\big)$ and taking
$s=\gamma^{-1}\left(\frac{1}{4\delta^{n_0+1}}\right)$ in the
inequality above, we get that
\begin{equation}\label{p2-1-12}
\begin{split}
J_{n_0}\le
\gamma^{-1}\left(\frac{1}{4\delta^{n_0+1}}\right)D^V(f,f)+
\beta\left(\gamma^{-1}\left(\frac{1}{4\delta^{n_0+1}}\right)\right)
\mu(\phi_1 f)^2+\frac{1}{4}.
\end{split}
\end{equation}
According to \eqref{aaa}, (\ref{p2-1-11a}) and (\ref{p2-1-12}), we
obtain
\begin{equation*}
\begin{split}
\mu(f^2)\le& \left(\frac{2\delta(\sqrt{\delta}+1)s_{n_0}}{\sqrt{\delta}-1}
+\gamma^{-1}\left(\frac{1}{4\delta^{n_0+1}}\right)\right)D^V(f,f)\\
&+
\beta\left(\gamma^{-1}\left(\frac{1}{4\delta^{n_0+1}}\right)\right)
\mu(\phi_1 f)^2+\frac{1}{2}.
\end{split}
\end{equation*}
Since $\mu(f^2)=1$,  this implies that
\begin{equation*}
\begin{split}
\mu(f^2)&\le \left(\frac{4\delta(\sqrt{\delta}+1)s_{n_0}}{\sqrt{\delta}-1}
+2\gamma^{-1}\left(\frac{1}{4\delta^{n_0+1}}\right)\right)D^V(f,f)\\
&
\quad +
2\beta\left(\gamma^{-1}\left(\frac{1}{4\delta^{n_0+1}}\right)\right)
\mu(\phi_1 f)^2
\end{split}
\end{equation*}
Hence, for $s>0$ small enough, we arrive at the desired super
Poincar\'{e} inequality by taking $n_0$ to be $n_0(s)$ defined by
(\ref{p2-1-8a}). The intrinsic ultracontractivity of  $(T_t^V)_{t
\ge 0}$ is easily verified by following the argument of (i).
\end{proof}

\begin{remark}
To derive the intrinsic ultracontractivity of Feynman-Kac semigroups (or Dirichlet semigroups), Rosen's lemma in the context of super log-Sobolev inequality was applied in \cite{CS,DS}.
Instead of such approach, in Theorem \ref{p2-1} we use super Poincar\'e inequality. The main advantage of our method is due to that,
 a mixed type super Poincar\'e inequality can be applied to study the situation that Sobolev inequality (\ref{p2-2-1}) fails, e.g.\
a symmetric $\alpha$-stable process on $\R$ with $\alpha \ge 1$, for which the method of \cite{CS} does not work.
\end{remark}

In Theorem \ref{p2-1}, we essentially use the lower bound estimate for the ground state. On the contrary, at the end of this section we present the following sufficient conditions for a upper bound estimate of the ground state.

\begin{proposition}\label{pro--00}  Suppose that the semigroup
$(T^V_t)_{t \ge 0}$ is intrinsically ultracontractive. Let $(L^V, \mathscr{D}(L^V))$ be the generator associated with $(T^V_t)_{t\ge0}$. If there
exist a positive function $\psi \in C_b^2(\R^d)\bigcap$ $ L^2(\R^d;dx)$ and a constant
$\lambda>0$ such that $\psi \in \D(L^V)$,
\begin{equation}\label{t2-1-0}
L^V\psi(x)\le \lambda \psi(x),\quad  x \in \R^d,
\end{equation}
then there is a constant $c_1>0$ such that \begin{equation*}\label{t2-1-0a}
\phi_1(x)\le c_1\psi(x),\quad x \in \R^d.
\end{equation*}

\end{proposition}

\begin{proof}
Under \eqref{t2-1-0}, we know that $$T_t^V\psi(x)\le e^{\lambda t}\psi(x),\quad x\in\R^d, t>0.$$
According to \cite[Theorem 3.2]{DS},   the intrinsic ultracontractivity of
$(T_t^V)_{t \ge 0}$ implies that for every $t>0$, there is a constant $c_t>0$ such that
\begin{equation*}
p^V(t,x,y)\ge c_t\phi_1(x)\phi_1(y),\quad \ x,y\in \R^d.
\end{equation*}
Therefore,
\begin{equation*}
\begin{split}\psi(x)&\ge e^{-\lambda} T_1^V\psi(x)= e^{-\lambda}  \int p^V(1,x,y)\psi(y)\,dy\\
&\ge c_1e^{-\lambda} \int \psi(y)\phi_1(y)\,dy\phi_1(x)=:C_0\phi_1(x),\end{split}
\end{equation*} which yields the required assertion.\end{proof}

\section{Intrinsic Ultracontractivity for Non-local Dirichlet Forms}\label{section3}
In this section, we come back to the framework introduced in Subsection \ref{subsection1-1} which is recalled below.  Let $(D,\mathscr{D}(D))$ be the non-local Dirichlet form given in \eqref{non-local} such that jump kernel satisfies (\ref{e3-1}) and
(\ref{e3-2}),
i.e., there exist $\alpha_1, \alpha_2\in(0,2)$ with $\alpha_1\le\alpha_2$
and positive $c_1, c_2, \kappa$ such that
\begin{equation*}
c_1|x-y|^{-d-\alpha_1}\le J(x,y)\le c_2|x-y|^{-d-\alpha_2},\quad
0<|x-y|\le \kappa
\end{equation*} and
\begin{equation*}
\sup_{x \in \R^d}\int_{\{|y-x|>\kappa\}}J(x,y)\,dy<\infty.\end{equation*}
Assume that the corresponding symmetric Hunt process $((X_t)_{t\ge0}, \Pp^x)$ is well defined for all $x\in\R^d$, and that the process $(X_t)_{t\ge0}$ possesses a  positive, bounded and continuous density function $p(t,x,y)$ on $\R^d\times \R^d$ for all $t>0$. That is, assumption {\bf(A1)} holds. Note that, when $J(x,y)=0$ for any $x$, $y\in\R^d$ with $|x-y|>
\kappa$, the process $(X_t)_{t\ge0}$ has finite range jumps.

Let $(T_t)_{t\ge0}$ be a symmetric semigroup associated with $(D,\mathscr{D}(D))$, i.e.\
 \begin{equation*} T_t(f)(x)=\Ee^x\left(f(X_t)\right),\,\, x\in\R^d,
f\in L^2(\R^d;dx).\end{equation*} Let $V$ be a non-negative measurable and locally bounded potential
function on $\R^d$ such that Assumption {\bf (A2)} is satisfied. Define the Feynman-Kac semigroup
$(T^V_t)_{t\ge0}$ associated with the Hunt process $(X_t)_{t \ge 0}$ as
follows:
\begin{equation*} T^V_t(f)(x)=\Ee^x\left(\exp\Big(-\int_0^tV(X_s)\,ds\Big)f(X_t)\right),\,\, x\in\R^d,
f\in L^2(\R^d;dx).\end{equation*} Then, the non-local regular Dirichlet form associated with $(T_t^V)_{t\ge0}$ is given by
\begin{equation*}
\begin{split}
D^V(f,f)&=\frac{1}{2} \iint\big(f(x)-f(y)\big)^2J(x,y)\,dx\,dy+\int
f^2(x)V(x)\,dx, \\
\mathscr{D}(D^V)&= \overline{C_c^1(\R^d)}^{{D_1^V}},
\end{split}
\end{equation*}
where ${D_1^V}(f,f):={D^V(f,f)+\|f\|_{L^2(\R^d;dx)}^2}.$  According to Propositions \ref{p1-1} and \ref{p1-2}, we know that the Feynman-Kac semigroup $(T_t^V)_{t\ge 0}$ is compact on $L^2(\R^d,dx)$, and it has a bounded, continuous and strictly positive ground state $\phi_1$ corresponding to the first eigenvalue $\lambda_1>0$.

\ \

In the following, we will apply Theorem \ref{p2-1} to establish the intrinsic ultracontractivity of Feynman-Kac semigroup $(T_t^V)_{t\ge 0}$. For this, we need
obtain some nice lower bound estimate for the ground state
$\phi_1$, and derive local super Poincar\'{e} inequality or Sobolev inequality for the Dirichlet form $(D^V,\mathscr{D}(D^V))$. These will be considered in the following three subsections respectively, and then general results about the intrinsic ultracontractivity of Feynman-Kac semigroup for non-local Dirichlet forms are presented in Subsection \ref{subsection3-4}.

\subsection{Lower bound estimate for the ground state}
For any Borel set $D \subseteq \R^d$, let
$\tau_D:=\inf\{t>0:\ X_t\notin D\}$ be the first exit time from $D$ of the process $(X_t)_{t\ge0}$. Denote by $B(x,r)$ the ball with center at $x\in\R^d$ and radius $r>0$.
\begin{lemma}\label{l3-1}
There exist positive constants $c_0:=c_0(\kappa)$ and
$r_0:=r_0(\kappa)$ such that for every $r \in (0,r_0]$ and $x \in
\R^d$, we have
\begin{equation}\label{l3-1-1}
\Pp^x\left(\tau_{B(x,r)}>c_0
r^{\alpha_2+\frac{(\alpha_2-\alpha_1)d}{\alpha_1}} \right)\ge \frac{1}{2}.
\end{equation}
\end{lemma}
\begin{proof}
For any $0<s<\kappa$, set
\begin{equation*}
\begin{split}
& L_1(s):=\sup_{x \in \R^d}\int_{\{|y-x|>s\}}J(x,y)\,dy,\\
& L_2(s):=\sup_{x \in \R^d}\int_{\{|y-x|\le s\}}|x-y|^2J(x,y)\,dy,\\
& L(s):=L_1(s)+ s^d
\big(s^{-2}L_2(s)\big)^{\frac{(d+\alpha_1)}{\alpha_1}}.
\end{split}
\end{equation*}
According to \cite[Theorem 2.1]{BKK}, there exists a constant
$r_0:=r_0(\kappa)>0$ such that for every $0<r<r_0$, $t>0$ and $x \in
\R^d$,
\begin{equation}\label{l3-1-2}
\Pp^x \left(\tau_{B(x,r)}<t\right)\le C_1 t L(r),
\end{equation}
where $C_1$ is a positive constant independent of $t$ and $r$.

Without lose of generality, we may and do assume that
$0<r_0<1$. Then, by (\ref{e3-1}) and (\ref{e3-2}), for every
$r \in (0,r_0)$
\begin{equation*}
\begin{split}
L(r)\le C_2\Big(r^{-\alpha_2-\frac{(\alpha_2-\alpha_1)d}{\alpha_1}}+
L_1(\kappa)\Big)\le C_3r^{-\alpha_2-\frac{(\alpha_2-\alpha_1)d}{\alpha_1}}.
\end{split}
\end{equation*}
 Let $c_0:=c_0(\kappa)$ be a positive constant such that $c_0C_1C_3\le
 {1}/{2}$. Then the required assertion \eqref{l3-1-1} follows from
 \eqref{l3-1-2} by taking $t=c_0r^{\alpha_2+\frac{(\alpha_2-\alpha_1)d}{\alpha_1}}.$
\end{proof}

\begin{lemma}\label{l3-2}
Let $r_0$, $c_0$ be the constants given in Lemma $\ref{l3-1}$, and
set $\varepsilon_0=r_0/\kappa$. Then, for any
$\varepsilon\in(0,\varepsilon_0)$, any two disjoint sets $B\supseteq
B(x,\varepsilon\kappa)$ and $D=B(y,\varepsilon\kappa)$ for some $x$,
$y \in \R^d$ satisfying that $dist(B,D)>{\varepsilon\kappa}$ and
$|z_1-z_2|\le \kappa$ for every $z_1 \in B$ and $z_2 \in D$, and
every $0< t_1<t_2<T(\kappa,\varepsilon):=c_0(\varepsilon
\kappa)^{\alpha_2+\frac{(\alpha_2-\alpha_1)d}{\alpha_1}},$ it holds that
\begin{equation*}\label{l3-2-1}
\Pp^x\Big(X_{\tau_B}\in D,t_1\le\tau_B<t_2\Big)\ge c_1\varepsilon^d\kappa^{-\alpha_1}(t_2-t_1),
\end{equation*}
where $c_1$ is a positive constant independent of
$\kappa$, $\varepsilon$, $x$ and $y$.
\end{lemma}
\begin{proof}
Denote by $p_B(t,x,y)$ the density of the process $(X_t)_{t\ge0}$ killed on exiting the set $B$, i.e.\
$$p_B(t,x,y)=p(t,x,y)-\Ee^x(\tau_B\le t; p(t-\tau_B, X(\tau_B),y)).$$
 According to the framework of the L\'evy system for the Dirichlet form $(D,\D(D))$ (see e.g. \cite[Lemma 4.8]{CK1}), we have for disjoint open
sets $B$ and $D$ that
 \begin{equation*}
\begin{split}\Pp^x\Big(X_{\tau_B}\in D\Big)&=\Ee^x\Big(\int_{0}^{\tau_B}\int_{D}J(X_s,z)\,dz\,ds\Big)\\
&=\int_B\int_{0}^{\infty}p_B(s,x,y)\,ds\int_{D}J(y,z)\,dz\,dy,\quad x\in B.\end{split}
 \end{equation*}
Then, following the proof of \cite[Proposition 2.5]{KS}, we get that
\begin{equation*}\label{l3-2-2}
\Pp^x\Big(X_{\tau_B}\in D,t_1\le\tau_B<t_2\Big)=\int_B\int_{t_1}^{t_2}p_B(s,x,y)\,ds\int_{D}J(y,z)\,dz\,dy.
\end{equation*}

Therefore, it holds for every
$\varepsilon \in (0,\varepsilon_0)$,
\begin{equation*}
\begin{split}
\Pp^x\Big(X_{\tau_B}\in D,t_1\le\tau_B<t_2\Big)&=\int_B\int_{t_1}^{t_2}p_B(s,x,y)\,ds\int_{D}J(y,z)\,dz\,dy\\
&\ge
\int_B\int_{t_1}^{t_2}p_B(s,x,y)\,ds\int_{D}\frac{c_1}{|z-y|^{d+\alpha_1}}\,dz\,dy\\
&\ge { C}\kappa^{-d-\alpha_1}|D|\int_B\int_{t_1}^{t_2}p_B(s,x,y)\,ds\,dy\\
&=
C\varepsilon^d\kappa^{-\alpha_1}\int_{t_1}^{t_2} \Pp^x\big(\tau_B>s\big)\,ds\\
&\ge C\varepsilon^d\kappa^{-\alpha_1}(t_2-t_1)\Pp^x\big(\tau_B>T(\kappa,\varepsilon)\big)\\
&\ge C\varepsilon^d\kappa^{-\alpha_1}(t_2-t_1)\Pp^x\big(\tau_{B(x,\varepsilon\kappa)}>T(\kappa,\varepsilon)\big)\\
&\ge {C\varepsilon^d}{\kappa^{-\alpha_1}}(t_2-t_1),
\end{split}
\end{equation*}
where in the first inequality we have used \eqref{e3-1}, the second inequality is
due to $|y-z|\le \kappa$ for every $y \in B$ and $z \in D$, and the last inequality follows from \eqref{l3-1-1} with $r={\varepsilon\kappa}$.
\end{proof}

\begin{lemma}\label{l3-3} Let $\varepsilon_0$, $c_0$ be the two constants
given in Lemma $\ref{l3-2}$. For any $\varepsilon\in
(0,\min({1}/{11},\varepsilon_0))$, let $D=B(0,2\varepsilon\kappa)$,
and $t_0=T(\kappa,\varepsilon):=c_0(\varepsilon
\kappa)^{\alpha_2+\frac{(\alpha_2-\alpha_1)d}{\alpha_1}}.$ Then, there
is a constant $c_2(\kappa,\varepsilon)>0$ such that for all
$x\in\R^d$ with $|x|> \frac{\kappa
(1-5\varepsilon)(1-4\varepsilon)}{\varepsilon}$,
\begin{equation}\label{l3-3-1}
T_{t_0}^V(\I_D)(x)\ge
\exp\Big(-\frac{1}{(1-6\varepsilon)\kappa}|x|\log\big(1+|x|+\sup_{|z|\le
|x|+2\varepsilon\kappa}V(z)\big)-c_{2}(\kappa,\varepsilon)\Big).
\end{equation}
\end{lemma}
\begin{proof}
For any $x\in\R^d$ with $|x|>\frac{\kappa (1-5\varepsilon)(1-4\varepsilon)}{\varepsilon}$, let $$n =\bigg\lfloor\frac{1}{(1-4\varepsilon)\kappa}|x|\bigg\rfloor+1$$ and
$x_i={ix}/{n}$ for any $0\le i\le n$, where $\lfloor x\rfloor$ denotes the largest integer less than or equal to $x$. In particular, $x_0=0$, $x_n=x$ and
 $$\frac{1}{(1-4\varepsilon)\kappa}|x|\le n < \frac{1}{(1-5\varepsilon)\kappa}|x|.$$Next, for all $0 \le i \le n$,
 set $D_i:=B(x_i,2\varepsilon{\kappa})$, $\tilde D_i:=B(x_i,\varepsilon\kappa)$. We can check that for all $0\le i\le n-1$, $dist(D_i,D_{i+1})>(1-5\varepsilon)\kappa-4\varepsilon\kappa\ge2\varepsilon{\kappa}$, and
$|z_i-z_{i+1}|\le (1-4\varepsilon)\kappa+4\varepsilon\kappa=\kappa$ for every $z_i \in D_i$ and $z_{i+1}\in
D_{i+1}$.

In the following, we define for all $n\ge 1$,
\begin{equation*}
\begin{split}
\tilde{\tau}_{D_i}:&=\inf\{t\ge \tilde{\tau}_{D_{i+1}}: X_t\notin D_i\},\quad 1\le i\le n-1; \\
\tilde{\tau}_{D_n}:&={\tau}_{D_n}.
\end{split}
\end{equation*}
By the convention, we also set $\tilde{\tau}_{D_{n+1}}=0$. Then,
\begin{equation}\label{l3-3-2}
\begin{split}
&T_{t_0}^V(\I_D)(x)\\
&=\Ee^x\Big(\I_D(X_{t_0})\exp\Big(-\int_0^{t_0} V(X_s)\,ds\Big)\Big)\\
& \ge \Ee^x\Big(0<\tilde{\tau}_{D_{i}}-\tilde{\tau}_{D_{i+1}}<\frac{t_0}{n},
X_{\tilde{\tau}_{D_i}}\in \tilde D_{i-1}\
{\rm for\ each}\ 1\le i \le n,\forall_{s\in[\tilde{\tau}_{D_1},t_0] } X_{s} \in D; \\
&\qquad\quad
\exp\Big(-\sum_{i=1}^n\int^{\tilde{\tau}_{D_{i}}}_{\tilde{\tau}_{D_{i+1}}}
V(X_s)\,ds-\int_{\tilde{\tau}_{D_1}}^{t_0}
 V(X_s)\,ds\Big)\Big)\\
& = \Ee^x\Big(0<{\tau}_{D_{n}}<\frac{t_0}{n}, X_{{\tau}_{D_{n}}}\in
\tilde D_{n-1}; \exp\Big(-\int^{{\tau}_{D_{n}}}_{0} V(X_s)\,ds\Big)\\
&\qquad\quad\cdot\Ee^{X_{\tilde{\tau}_{D_{n}}}}\Big(0<\tau_{D_{n-1}}<\frac{t_0}{n},X_{{\tau}_{D_{n-1}}}\in
\tilde D_{n-2}; \exp\Big(-\int^{{\tau}_{D_{n-1}}}_{0} V(X_s)\,ds\Big)\\
&\qquad\quad\,\cdot \Ee^{X_{\tilde{\tau}_{D_{n-1}}}}\Big(\cdots
\Ee^{X_{\tilde{\tau}_{D_{2}}}}\Big(0<{\tau}_{D_{1}}<\frac{t_0}{n},X_{\tau_{D_{1}}}\in
\tilde D_0; \exp\Big(-\int^{{\tau}_{D_{1}}}_{0} V(X_s)\,ds\Big)\\
&\qquad\quad\,\,\cdot
\Ee^{X_{\tilde{\tau}_{D_{1}}}}\Big(\forall_{s\in[0,t_0-\tilde {\tau}_{D_1}]} X_{s}
\in D;
\exp\Big(-\int_{0}^{t_0-\tilde {\tau}_{D_{1}}}V(X_s)\,ds\Big)\Big)\Big)\cdots\Big)\Big)\Big),
\end{split}
\end{equation}
where in the last equality we have used the strong Markov property.

On the one hand, according to Lemma \ref{l3-2}, for any
$2\le i\le n+1$, if $X_{\tilde{\tau}_{D_{i}}} \in \tilde D_{i-1}$, then for every $i>1$,
 \begin{align*}
&
\Ee^{X_{\tilde{\tau}_{D_{i}}}}\Big(0<{\tau}_{D_{i-1}}<\frac{t_0}{n},X_{{\tau}_{D_{i-1}}}\in
\tilde D_{i-2}; \exp\Big(-\int^{{\tau}_{D_{i-1}}}_{0} V(X_s)\,ds\Big)\Big)\\
&\ge
\sum_{j=1}^{\infty}\Ee^{X_{{\tau}_{D_{i}}}}\Big(\frac{t_0}{(j+1)n}\le
{\tau}_{D_{i-1}}<\frac{t_0}{jn},X_{\tau_{D_{i-1}}}\in
\tilde D_{i-2}; \exp\Big(-\int^{{\tau}_{D_{i-1}}}_{0} V(X_s)\,ds\Big)\Big)\\
& \ge \sum_{j=1}^{\infty}\exp\Big(-\frac{t_0}{jn}\sup_{x \in
D_{i-1}}V(x)\Big)\\
&\qquad\qquad \times\inf_{y \in \tilde D_{i-1}}
\Ee^y\Big(\frac{t_0}{(j+1)n}\le{\tau}_{D_{i-1}}<\frac{t_0}{jn},
X_{{\tau}_{D_{i-1}}}\in
\tilde D_{i-2}\Big)\\
&\ge
\frac{C\varepsilon^d\kappa^{-\alpha_1} t_0}{n}\sum_{j=1}^{\infty}\frac{1}{j(j+1)}\exp\Big(-\frac{t_0}{jn}\sup_{x
\in D_{i-1}}V(x)\Big)\\
& \ge \frac{C\varepsilon^d\kappa^{-\alpha_1}t_0 }{n+t_0\sup_{x \in D_{i-1}}V(x)},
 \end{align*}
where in the third inequality we have used Lemma \ref{l3-2} with
$B=D_{i-1}$ and $D=\tilde D_{i-2}$, and the last inequality follows
from \cite[Lemma 5.2]{KS}, i.e.\
\begin{equation}\label{rrr1}\sum_{j=1}^\infty \frac{e^{-r/j}}{j(j+1)}\ge \frac{e^{-1}}{r+1},\quad r\ge 0.\end{equation}

On the other hand, due to Lemma \ref{l3-1}, if $X_{\tilde{\tau}_{D_1}}\in
\tilde D_0$, then
\begin{equation}\label{l3-3-3}
\begin{split}
&\Ee^{X_{\tilde{\tau}_{D_{1}}}}\Big(\forall_{s\in  [0,t_0-\tilde {\tau}_{D_1}]} X_{s}
\in D;
\exp\Big(-\int_{0}^{t_0-\tilde {\tau}_{D_{1}}}V(X_s)\,ds\Big)\Big)\\
&\ge \exp\Big(-t_0\sup_{z\in D}V(z)\Big)\inf_{y \in \tilde
D_0}\Ee^y\Big (\tau_D>t_0\Big)\\
&\ge \exp\Big(-t_0\sup_{z\in D}V(z)\Big)\inf_{y \in \tilde
D_0}\Ee^y\Big
(\tau_{B(y,\varepsilon\kappa)}>T(\kappa,\varepsilon)\Big)\\
& \ge {C}(\kappa,\varepsilon).
\end{split}
\end{equation}

Combining all the estimates above with the fact that $n \le\frac{1}{(1-5\varepsilon)\kappa}|x|$, we obtain that
\begin{equation*}
\begin{split}
T_{t_0}^V(\I_D)(x)&  \ge C(\kappa,\varepsilon) \prod_{i=1}^n\Big(\frac{c\varepsilon^d\kappa^{-\alpha_1} t_0}{n+t_0\sup_{z
\in D_{i-1}}V(z)}\Big)\\
&\ge C(\kappa,\varepsilon)\Big(\frac{c\varepsilon^d\kappa^{-\alpha_1} t_0}{n+t_0\sup_{|z|\le |x|+2\varepsilon\kappa}V(z)}\Big)^n\\
&\ge \exp\Big(-\frac{1}{(1-6\varepsilon)\kappa}|x|\log(1+|x|+\sup_{|z|\le |x|+2\varepsilon\kappa}V(z))-C(\kappa,\varepsilon)\Big),
\end{split}
\end{equation*} which completes the proof.
\end{proof}
According to Lemma \ref{l3-3}, we can obtain the following lower bound estimate for the ground state.
\begin{proposition}\label{p3-1} For any $\varepsilon\in (0,\min({1}/{11},\varepsilon_0))$ and  $x\in \R^d$, it holds
\begin{equation}\label{bbb}
\phi_1(x) \ge \exp\Big(-\frac{1}{\kappa(1-6\varepsilon)}|x|\log(1+|x|+\sup_{|z|\le |x|+2\varepsilon\kappa}V(z))-c_3(\kappa,\varepsilon)\Big)
\end{equation}
for some positive constant $c_3(\kappa,\varepsilon)$ independent of $x$.
\end{proposition}
\begin{proof}  Since $\phi_1$ is continuous and strictly positive, we only need to verify the desired assertion for $x\in\R^d$ with $|x|>\frac{\kappa (1-5\varepsilon)(1-4\varepsilon)}{\varepsilon}$.
According to (\ref{l3-3-1}), we have for any $x\in\R^d$ with $|x|>\frac{\kappa (1-5\varepsilon)(1-4\varepsilon)}{\varepsilon}$ that
\begin{equation*}
\begin{split}
\exp\Big(-\frac{1}{\kappa(1-6\varepsilon)}|x|\log
\big(1+|x|&+\sup_{|z|\le |x|+2\varepsilon\kappa}V(z)\big)-C(\kappa,\varepsilon)\Big)\\
&\le T_{t_0}^V(\I_D)(x)\le  cT_{t_0}^V(\phi_1)(x)=ce^{-\lambda_1 t_0}\phi_1(x),
\end{split}
\end{equation*}
where $c:=(\inf_{y \in D}\phi_1(y))^{-1}<\infty$. This immediately yields the desired assertion.
\end{proof}
\subsection{Intrinsic local super Poincar\'e inequality}
In this part, we will present the following local intrinsic super
Poincar\'e inequality.
\begin{proposition}\label{l3-4}
Let $\varphi$ be a positive and continuous function
 on $\R^d$. For any
$r\ge\kappa$, $s>0$ and $f\in C_c^2(\R^d)$,
\begin{equation*}
\begin{split}
 \int_{B(0,r)} f^2(x) \,dx
\le & s D^V(f,f)\\
&+\frac{c(\kappa)}{\inf_{|x|\le r+\kappa} \varphi^2(x)}
 \big(1+s^{-\frac{d}{\alpha_1}}\big)\Big(\int_{B(0,r+\kappa)}|f(x)|\varphi(x)\,dx\Big)^2.
\end{split}
\end{equation*}
\end{proposition}
\begin{proof}
(i) According to \eqref{e3-1} and the fact that $V \ge 0$, for any $f\in C_c^2(\R^d)$,
\begin{equation}\label{pro1-00}
\begin{split}
D_{\alpha_1,\kappa}(f,f)&:=c_1\iint_{\{|x-y|\le \kappa\}}\big(
f(x)-f(y)\big)^2 |x-y|^{-d-\alpha_1}\,dx\,dy\\
&\le D^V(f,f).
\end{split}
\end{equation}
Next, we follow the argument of \cite[Theorem 3.1]{CK1} to obtain that for any $s>0$, $r\ge \kappa$ and $f\in C_c^2(\R^d)$,
\begin{equation}\label{pro-1-01}
\begin{split}
 \int_{B(0,r)} f^2(x) \,dx  \le&  s D_{\alpha_1,\kappa}(f,f)+C(\kappa)
 \big(1+s^{-\frac{d}{\alpha_1}}\big)\Big(\int_{B(0,r+\kappa)}|f(x)|\,dx\Big)^2
\end{split}
\end{equation}
holds with some constant $C(\kappa)>0$. If \eqref{pro-1-01} holds,
then, combining it with (\ref{pro1-00}) above will complete the proof.

(ii) Next, we turn to the proof of \eqref{pro-1-01}. For any $0<s\le r$
and $f\in C_c^2(\R^d)$, define
$$f_s(x):=\frac{1}{|B(0,s)|}\int_{B(x,s)}f(z)\,dz,\quad x\in B(0,r).$$ We have
$$\sup_{x\in B(0,r)}|f_s(x)|\le \frac{1}{|B(0,s)|} \int_{B(0,r+s)}|f(z)|\,dz,$$
and $$\aligned \int_{B(0,r)}|f_s(x)|\,dx&\le \int_{B(0,r)}\frac{1}{|B(0,s)|}\int_{B(x,s)}|f(z)|\,dz\,dx\\
&\le
\int_{B(0,r+s)}\bigg(\frac{1}{|B(0,s)|}\int_{B(z,s)}\,dx\bigg)|f(z)|\,dz\\
&\le
\int_{B(0,r+s)}|f(z)|\,dz.
\endaligned$$ Thus,
$$\aligned\int_{B(0,r)}f_s^2(x)\,dx\le & \Big(\sup_{x\in B(0,r)}|f_s(x)|\Big) \int_{B(0,r)}|f_s(x)|\,dx\\
\le &\frac{1}{|B(0,s)|}
\bigg(\int_{B(0,r+s)}|f(z)|\,dz\bigg)^2.\endaligned$$

Therefore, for any $f\in C_c^2(\R^d)$ and $0<s\le r,$
$$\aligned\int_{B(0,r)}f^2(x)\,dx
\le & 2\int_{B(0,r)}\big(f(x)-f_s(x)\big)^2\,dx+ 2\int_{B(0,r)}f^2_s(x)\,dx\\
\le &2\int_{B(0,r)}\frac{1}{|B(0,s)|}\int_{B(x,s)}(f(x)-f(y))^2\,dx\,dy\\
&+ \frac{2}{|B(0,s)|} \bigg(\int_{B(0,r+s)}|f(z)|\,dz\bigg)^2\\
\le & \bigg(\frac{2s^{d+\alpha_1}}{|B(0,s)|}\bigg)\int\int_{\{|x-y|\le s\}}\frac{(f(x)-f(y))^2}{|x-y|^{d+\alpha_1}}\,dx\,dy\\
&+ \frac{2}{|B(0,s)|} \bigg(\int_{B(0,r+s)}|f(z)|\,dz\bigg)^2.\endaligned$$

In particular, for any $f\in C_c^2(\R^d)$ and $0<s\le \kappa\le r,$
$$\aligned\int_{B(0,r)}f^2(x)\,dx\le &c_3\bigg[s^{\alpha_1}D_{\alpha_1,\kappa}(f,f)+
s^{-d}\bigg(\int_{B(0,r+\kappa)}|f(z)|\,dz\bigg)^2\bigg], \endaligned$$
which implies that there exists a constant $s_0:=s_0(\kappa)>0$ such
that for all $s\in(0,s_0]$, $$\aligned\int_{B(0,r)}f^2(x)\,dx\le &s
D_{\alpha_1,\kappa}(f,f)\\
&+
c_4s^{-d/\alpha_1}\bigg(\int_{B(0,r+\kappa)}|f(z)|\,dz\bigg)^2,\quad
r\ge\kappa,f\in C_c^2(\R^d). \endaligned$$ This proves the desired
assertion \eqref{pro-1-01}.
\end{proof}

\begin{remark} One also can derive the inequality \eqref{pro-1-01} along the lines of the proof of \cite[Lemma 2.1]{CW14}. Indeed, when $\kappa=1$, the inequality \eqref{pro-1-01} is just
\cite[Lemma 2.1]{CW14}. For general $\kappa>0$, the proof of \eqref{pro-1-01} is almost the same as that of \cite[Lemma 2.1]{CW14}, and one only need to replace $B(0,\frac{1}{2})$ in \cite[(2.17)]{CW14} by  $B(0,\frac{\kappa}{2})$.\end{remark}

\subsection{Sobolev inequalities}

\begin{proposition}\label{ppp-444} If $d>\alpha_1$, then there exists a constant $c_0(\kappa)>0$ such that the following Sobolev inequality holds
\begin{equation}\label{e3-3}
\begin{split}
\|f\|_{L^{2d/(d-\alpha_1)}(\R^d;dx)}^2 & \le c_0(\kappa)\bigg[D(f,f)+\|f\|_{L^{2}(\R^d;dx)}^2\bigg],\quad f \in C_c^{\infty}(\R^d).
\end{split}
\end{equation}
\end{proposition}
\begin{proof} According to \cite[(2.3)]{CK} and (\ref{e3-1}), we get that for any $f\in C_c^{\infty}(\R^d)$, $$\aligned
\|f\|_{L^{{2d}/({d-\alpha_1})}(\R^d;dx)}^2 &\le c_5
\int_{\R^d}\int_{\R^d}\frac{\left(f(x)-f(y)\right)^2}{|x-y|^{d+\alpha_1}}dxdy\\
& \le c_5
\int\int_{\{|x-y|\le \kappa\}}\frac{\left(f(x)-f(y)\right)^2}{|x-y|^{d+\alpha_1}}dxdy
+c_6(\kappa)\|f\|_{L^{2}(\R^d;dx)}^2\\
&\le c_1c_5 D(f,f)+c_6(\kappa)\|f\|_{L^{2}(\R^d;dx)}^2.
\endaligned$$
This completes the proof.
\end{proof}
\subsection{General results about intrinsic ultracontractivity of Feynman-Kac
semigroups}\label{subsection3-4}
According to Propositions \ref{p3-1},  \ref{l3-4}, \ref{ppp-444} and
Theorem \ref{p2-1}, we immediately have the following statement.
\begin{theorem}\label{t3-1} Let $(T_t^V)_{t\ge 0}$ be a compact Feynman-Kac semigroup on $L^2(\R^d,dx)$ given in the beginning of this section (or in Subsection \ref{subsection1-1}), and $V$ be a locally bounded non-negative measurable function on $\R^d$ such that assumptions {\bf (A2)} and {\bf (A4)} hold.
For some  $\varepsilon\in
(0,\min({1}/{11},\varepsilon_0))$, define
\begin{equation}\label{t3-1-1}
\varphi(x):=\exp\Big(-\frac{1}{\kappa(1-6\varepsilon)}|x|\log(1+|x|+\sup_{|z|\le |x|+2\varepsilon\kappa}V(z))\Big),
\end{equation} and
\begin{equation*}\label{t3-1-3}
\begin{split}
\alpha(r,s):=\frac{ c(\kappa)}{\inf_{|x|\le r+\kappa}
\varphi^2(x)}  \big(1+s^{-\frac{d}{\alpha_1}}\big),
\end{split}
\end{equation*}  where $\varepsilon_0$ is the constant in Lemma $\ref{l3-2}$ and $c(\kappa)$ is the constant in Proposition
$\ref{l3-4}$.
For the constant $K$ in {\bf (A4)}, let \begin{equation*}\label{t3-1-3}
\begin{split}
\Phi(R):=&\inf_{x \in \R^d: |x|\ge R, V(x)>K}V(x),\\  \Theta(R):=&
\big|\{x \in \R^d: |x|\ge R, V(x)\le K\}\big|,\\
\Psi(R):=&\frac{1}{\Phi(R)}+c_0(\kappa)\Theta(R)^{{\alpha_1}/{d}},\end{split}
\end{equation*}where $c_0(\kappa)$ is a constant in \eqref{e3-3}.
We furthermore define
\begin{equation}\label{t3-1-3a}
\begin{split}\beta_\Phi(s):=&\left[1+\alpha\left(
\Phi^{-1}\left(\frac{2}{s}\right), \frac{s}{2}\right)\right],\\
 \beta_\Psi(s):=&\left[1+\alpha\left( \Psi^{-1}\left(\frac{s}{4}\right),
\frac{s}{4}\right)\right],\\
\gamma(s):=&\Theta^{-1}\left(\frac{s}{2}\right).
\end{split}\end{equation}
\begin{enumerate}
\item [(1)]
If $\lim_{|x|\to\infty}V(x)=\infty$, and
\begin{equation*}
\int_t^{\infty}\frac{\beta_\Phi^{-1}(s)}{s}\,ds<\infty,\quad t\gg1,
\end{equation*}
then $(T_t^V)_{t \ge 0}$ is intrinsically ultracontractive.

\item[(2)]  If $d>\alpha_1$ and
\begin{equation*}
\int_t^{\infty}\frac{\beta_\Psi^{-1}(s)}{s}\,ds<\infty,\quad t\gg1,
\end{equation*}
then $(T_t^V)_{t \ge 0}$ is intrinsically ultracontractive.

\item [(3)] Suppose that there exists a constant $\delta>1$ such that
\begin{equation}\label{t3-1-4}
\sum_{n=1}^{\infty}\gamma(s_n)\delta^n<\infty
\end{equation}
and
\begin{equation}\label{t3-1-5}
\int_t^{\infty}\frac{\tilde \beta^{-1}_\Phi(s)}{s}\,ds<\infty,\quad t\gg1,
\end{equation}
where $s_n:=\beta_\Phi^{-1}(\frac{c_1 \delta^n}{2})$ with
$c_1:=\|\phi_1\|_{\infty}^2$, and
 \begin{equation}\label{t3-1-5a}
 \tilde
\beta_\Phi(s):=2\beta_\Phi\left(\gamma^{-1}\left(\frac{1}{4\delta^{n_0(s)+1}}\right)\right)
\end{equation}
with
\begin{equation*}
\begin{split}
n_0(s):=\inf\Bigg\{N \ge1:\frac{4\delta(\sqrt{\delta}+1)s_{N}}{\sqrt{\delta}-1}
+2\gamma^{-1}\left(\frac{1}{4\delta^{N+1}}\right)\le s\Bigg\}.
\end{split}
\end{equation*}
Then $(T_t^V)_{t \ge 0}$ is intrinsically ultracontractive.
\end{enumerate}

\end{theorem}

\section{Proofs of Theorems and Examples}\label{section4}
In the section, we will give the proofs of all the statements in Section \ref{section1}. First, we present the

\begin{proof}[Proof of Theorem $\ref{thm2}$]
(1) It is clear that $$\lim_{|x|\to\infty}V(x)=\infty.$$  Let
$\varepsilon\in (0,\min({1}/{11},\varepsilon_0))$. Since
$$V(x)\le c_4|x|^{\theta_3}\log^{\theta_4}(1+|x|),\quad x\in\R^d,$$ we have the following estimate for the function
$\varphi$ given by (\ref{t3-1-1})
\begin{equation*}
\varphi(x)\ge \exp\Big(-\frac{\theta_3}{\kappa(1-7\varepsilon)}(1+|x|)\log(1+|x|)-
C(\kappa,\varepsilon,\theta_3,\theta_4)\Big).
\end{equation*}
On the other hand, since $$V(x)\ge
c_3|x|^{\theta_1}\log^{\theta_2}(1+|x|),\quad x\in\R^d,$$ we obtain
\begin{equation*}\label{t1-3}
\begin{split}
\Phi(r)\ge c_3r^{\theta_1}\log^{\theta_2}(1+r).
\end{split}
\end{equation*}
Therefore, the rate function
$\beta_\Phi(r)$ defined by (\ref{t3-1-3a}) satisfies that for $s>0$ small enough
\begin{equation*}
\beta_\Phi(s)\le  C(\kappa, \varepsilon)\exp\left\{C(\kappa,
\varepsilon,\theta_3,\theta_4)
\Big(1+s^{-\frac{1}{\theta_1}}\log^{1-\frac{\theta_2}{\theta_1}}\big(1+s^{-1}\big)\Big)\right\}.
\end{equation*}
In particular, $$\beta_\Phi^{-1}(r)\le \frac{C}{\log^{\theta_1}(1+
r)\log^{\theta_2-\theta_1}\log(e+r)},\quad r>0\textrm{ large enough}.$$
Then, if $\theta_1=1$ and $\theta_2>2$ or if $\theta_1>1$, the intrinsic ultracontractivity of
$(T_t^V)_{t \ge 0}$ immediately follows from
Theorem \ref{t3-1}(1).

(2)
The required lower bound for
the ground state $\phi_1$ immediately follows from \eqref{bbb}.
Next, we will verify the upper bound. If $\theta_1=1$ and $\theta_2>2$ or if $\theta_2>1$,
then the semigroup $(T_t^V)_{t\ge0}$ is intrinsically ultracontractive. For any $0<\lambda<\theta_1$, let
\begin{equation*}
\psi(x):=\exp\Big(-\frac{\lambda}{2\kappa} \sqrt{1+|x|^2}\log(1+|x|^2)\Big).
\end{equation*}

Suppose that assumption {\bf (A4)} holds and for any $x,y\in\R^d$ with $|x-y|>\kappa$, $J(x,y)=0$.
Then, the generator $L^V$ of the associated Feynman-Kac semigroup
$(T^V_t)_{t\ge0}$ enjoys the expression
(\ref{ope11}). By the approximation argument, it is easy to verify that $\psi \in \D(L^V)$. For
$|x|$ large enough,  we obtain by the mean value theorem that
\begin{equation*}
\begin{split}
 L^V \psi(x)&\le C_1 \sup_{z\in B(x,\kappa)}\Big(
\big|\nabla \psi(z)\big|+\big|\nabla^2 \psi(z)\big|\Big)-V(x)\psi(x)\\
&\le C_2(\kappa, \lambda)\log^2(1+|x|)
\cdot \exp\Big(-\frac{\lambda}{\kappa} (|x|-\kappa)\log\big(1+(|x|-\kappa)\big)-C_3(\kappa,\lambda)\Big)\\
&\quad -c_3(1+|x|)^{\theta_1}\log^{\theta_2}(1+|x|)\psi(x)\\
&\le C_4(\kappa, \lambda) (1+|x|)^{\lambda}\log^2(1+|x|)\psi(x)-c_3(1+|x|)^{\theta_1}\log^{\theta_2}(1+|x|)\psi(x).
\end{split}
\end{equation*}
Since $0<\lambda<\theta_1$, $$L^V \psi(x)\le 0$$ for $|x|$ large
enough. Note that the function $x\mapsto L^V \psi(x)$ is locally
bounded, we know from the inequality above that (\ref{t2-1-0}) holds
with some constant $\lambda>0$. Therefore, the required upper bound
for $\phi_1$ follows from Proposition \ref{pro--00}.
\end{proof}

Indeed, according to Theorem \ref{t3-1}(1), we also have the following statements. The proofs are similar to that of Theorem \ref{thm2}, and so we omit them here.

\begin{proposition}\label{pro} Suppose that \eqref{e3-1}, \eqref{e3-2}, assumptions {\bf(A1)} and {\bf (A2)} hold. Then, we have the following two assertions.
\begin{itemize}
\item[(1)] If there are positive constants $c_5$, $c_6$, $\theta_5$, $\theta_6$ with $\theta_5>\theta_6+1$ such that for all $x\in\R^d$,
$$ c_5(1+|x|^{\theta_5}) \le V(x)\le e^{c_6(1+|x|^{\theta_6})},$$
then $(T_t^V)_{t \ge 0}$ is intrinsically ultracontractive, and for any $\varepsilon>0$ there is a constant $C_3:=C_3(\varepsilon)>0$ such that for all $x\in\R^d$,
$$C_3\exp\Big(-\frac{(1+\varepsilon)}{\kappa} |x|^{\theta_6+1}\Big)\le
\phi_1(x).$$

Additionally, if moreover {\bf (A3)} also holds and
 $$J(x,y)=0,\quad x,y\in\R^d\textrm{ with } |x-y|>\kappa,$$ then for any $\varepsilon>0$, there exists a constant
$C_4:=C_4(\varepsilon)>0$ such that for all $x\in\R^d$,
$$
\phi_1(x) \le C_4\exp\Big(-\frac{(1-\varepsilon)  \theta_5}{\kappa}
|x|\log(1+|x|)\Big).
$$

\item [(2)] If there are positive constants $c_7$, $c_8$, $\theta_7$, $\theta_8$ with $\theta_7\le \theta_8$ such that for all $x\in\R^d$,
$$e^{c_7(1+|x|^{\theta_7})} \le V(x)\le e^{c_8(1+|x|^{\theta_8})},$$
then $(T_t^V)_{t \ge 0}$ is intrinsically ultracontractive, and for any $\varepsilon>0$ there is a constant $C_5:=C_5(\varepsilon)>0$ such that for all $x\in\R^d$,
$$C_5\exp\Big(-\frac{c_8(1+\varepsilon)}{\kappa}
|x|^{\theta_8+1}\Big)\le\phi_1(x).$$

Additionally, if moreover {\bf (A3)} also holds and
 $$J(x,y)=0,\quad x,y\in\R^d\textrm{ with } |x-y|>\kappa,$$ then for any $\frac{\theta_7}{\theta_7+1}<\varepsilon<1$,  there exists a constant
$C_6:=C_6(\varepsilon)>0$ such that for all $x\in\R^d$,
$$
\phi_1(x)\le
C_6\exp\Big(-\frac{c_7(1-\varepsilon)}{\kappa} |x|^{\theta_7+1}\Big).
$$

\end{itemize}
\end{proposition}

\ \

Next, we turn to the

\begin{proof}[Proof of Example $\ref{ex2}$]
(1) According to Theorem \ref{thm2}, we know
that $(T_t^V)_{t \ge 0}$ is intrinsically ultracontractive if $V(x)=|x|^{\theta}$ for some $\theta>1$. Now we are going to
verify that
if
$V(x)=|x|^{\theta}$ for some $0<\theta\le 1$, then $(T_t^V)_{t \ge 0}$ is not intrinsically ultracontractive.
We mainly use the method of  \cite[Theorem 1.6]{KS} (see \cite[pp. 5055-5056]{KS})
and disprove \cite[Condition 1.3, p.\ 5027]{KS}.
In fact, according to \cite[Condition 1.3, p.\ 5027]{KS}, if $(T_t^V)_{t \ge 0}$
is intrinsically ultracontractive, then for
every fixed $t \in (0,1]$, there exists a constant $C_t>0$ such that
\begin{equation}\label{ex2-0}
T_t^V(\I_D)(x)\ge C_t T_t^V(\I_{B(x,1)})(x).
\end{equation}

Let $p(t,x,y)$ be the heat kernel for the associated process $(X_t)_{t\ge0}$. According to \cite[(1.16) in Theorem 1.2 and (1.20) in Theorem 1.4]{CKK1}, for
any fixed $t\in(0,1]$ and $|x-y|$ large enough,
\begin{equation*}
p(t,x,y)\le  C_1t\exp\bigg(-C_2|x-y|\Big(\log\frac{|x-y|}{t}\Big)^{\frac{\gamma-1}{\gamma}}\bigg).
\end{equation*}
Set $D=B(0,1)$. For $|x|$ large enough,
\begin{equation*}
\begin{split}
& T_t^V(\I_D)(x)\le \int_D p(t,x,y)\,dy \le
C_3t\exp\Big(-C_2(|x|-1)\Big(\log\frac{|x|-1}{t}\Big)^{\frac{\gamma-1}{\gamma}}\Big).
\end{split}
\end{equation*}

On the other hand, for $|x|$ large enough,
\begin{equation*}
\begin{split}
 T_t^V(\I_{B(x,1)})(x)&\ge
\Ee^x\Big(\tau_{B(x,1)}>t;
\exp\Big(-\int_0^t V(X_s)ds\Big)\Big)\\
&\ge C\Pp^x\big(\tau_{B(x,1)}>t\big)e^{-t|x|^{\theta}}\\
&\ge C\Pp^x\big(\tau_{B(x,1)}>1\big)e^{-t|x|^{\theta}}\\
&\ge C e^{-t|x|^{\theta}}.
\end{split}
\end{equation*}
Combining both conclusions above with the fact that $\theta\in(0,1]$, we
get that for any fixed $t\in(0,1]$, the inequality (\ref{ex2-0}) does not hold for
any constant $C_t>0$,
which contradicts with \cite[Condition 1.3, p.\ 5027]{KS}. Hence, according to the remark below \cite[Condition 1.3, p.\ 5027]{KS}, the semigroup
$(T_t^V)_{t \ge 0}$ is not intrinsically ultracontractive.

(2) If $\gamma=\infty$ and $\theta>1$, then the ground state estimate (\ref{ex2-1}) immediately follows from
Theorem \ref{thm2}. When $1<\gamma<\infty$ and $\theta>1$, one can apply Proposition \ref{p3-1}
to get a lower bound estimate for $\phi_1$, which however is not optimal. Instead, we will adopt a slightly different argument
from that of
Proposition \ref{p3-1}, and will
derive a more accurate lower bound estimate, which is partly inspired by \cite[Theorem 5.4]{CKK1}.

For any $\lambda>0$, we choose a constant $$0<\varepsilon<\varepsilon_0\wedge\bigg(\frac{1}{2}\theta^{\frac{1}{\gamma}}
\Big(\big(1+\lambda)^{\frac{1}{\gamma}}-1\Big)\bigg), $$
where $\varepsilon_0>0$ is the same constant in
Lemma \ref{l3-2}.  For every $x \in \R^d$ with $$|x|\ge e^{2^\gamma\theta(1+2\varepsilon)^\gamma}\vee(e-1)\quad\textrm{ and  } \quad\theta^{-\frac{1}{\gamma}}|x|\log^{-\frac{1}{\gamma}}(1+|x|)\ge1,$$ let
$$n=\Big\lfloor\theta^{-\frac{1}{\gamma}}|x|\log^{-\frac{1}{\gamma}}(1+|x|)\Big\rfloor+1,$$
and $x_i:={i}x/n$  for any $0\le i\le n$, where $\lfloor x\rfloor$ denotes the largest integer less than or equal to $x$.  Next, for all $0\le i\le n$,
we set $D_i:=B(x_i, \varepsilon)$ and $\tilde D_i:=B(x_i, {\varepsilon}/{2})$.
Note that
\begin{equation}\label{ex2-3}
\theta^{-\frac{1}{\gamma}}|x|\log^{-\frac{1}{\gamma}}(1+|x|)\le n \le
\theta^{-\frac{1}{\gamma}}|x|\log^{-\frac{1}{\gamma}}(1+|x|)+1,
\end{equation}
we can check that  for each $0 \le i \le n-1$,
$$1\le \frac{1}{2}\theta^{\frac{1}{\gamma}}\log^{\frac{1}{\gamma}}(1+|x|)-
2\varepsilon\le\frac{1}{\theta^{-\frac{1}{\gamma}}\log^{-\frac{1}{\gamma}}(1+|x|)+1}-2\varepsilon\le dist(D_i,D_{i+1})$$
and for every $z_i \in D_i$ and $z_{i+1}\in D_{i+1}$,
$$|z_i-z_{i+1}|\le  \frac{|x|}{n}+2\varepsilon\le {\theta^{\frac{1}{\gamma}}\log^{\frac{1}{\gamma}}(1+|x|)}+2\varepsilon\le  \Big((1+\lambda)\theta\log(1+|x|)\Big)^{\frac{1}{\gamma}}.$$
In the following, we define for all $n\ge1$
\begin{equation*}
\begin{split}
\tilde{\tau}_{D_i}:&=\inf\{t\ge \tilde{\tau}_{D_{i+1}}: X_t\notin D_i\},\quad 1\le i\le n-1; \\
\tilde{\tau}_{D_n}:&={\tau}_{D_n}.
\end{split}
\end{equation*}
By the convention, we also set $\tilde{\tau}_{D_{n+1}}=0$.
Let $T(1,\varepsilon)$ be the same constant in Lemma \ref{l3-2}
with $\kappa=1$.

First, if $X_{\tilde \tau_{D_{i}}} \in \tilde D_{i-1}$, then we have for each $i\ge2$, $j\ge1$ and
$t_0=T(1,\varepsilon)$,
\begin{align*}
& \Pp^{X_{\tilde \tau_{D_{i}}}}\Big(\frac{t_0}{(j+1)n}\le \tau_{D_{i-1}}<\frac{t_0}{jn},X_{\tau_{D_{i-1}}}\in
\tilde D_{i-2}\Big)\\
&\ge \inf_{x \in \tilde D_{i-1}}\Pp^x\Big(\frac{t_0}{(j+1)n}\le \tau_{D_{i-1}}<\frac{t_0}{jn},X_{\tau_{D_{i-1}}}\in
\tilde D_{i-2}\Big)\\
&=\inf_{x \in \tilde D_{i-1}}\int_{\frac{t_0}{(j+1)n}}^{\frac{t_0}{jn}}\int_{D_{i-1}}
p_{D_{i-1}}(s,x,y)\int_{\tilde D_{i-2}}J(y,z)\,dz\,dy\,ds\\
&\ge C\inf_{x \in \tilde D_{i-1}}\int_{\frac{t_0}{(j+1)n}}^{\frac{t_0}{jn}}\int_{D_{i-1}}
p_{D_{i-1}}(s,x,y)\int_{\tilde D_{i-2}}e^{-|y-z|^{\gamma}}\,dz\,dy\,ds\\
&\ge C \inf_{x \in \tilde D_{i-1}}\int_{\frac{t_0}{(j+1)n}}^{\frac{t_0}{jn}}
\Pp^x\big(\tau_{D_{i-1}}>s\big)\,ds \, e^{-(1+\lambda)\big(\theta\log(1+|x|)\big)}\\
&\ge \frac{Ct_0}{j(j+1)n(1+|x|)^{(1+\lambda)\theta}}\inf_{x \in \tilde D_{i-1}}\Pp^x\Big(\tau_{D_{i-1}}>\frac{t_0}{jn}\Big)\\
&\ge \frac{Ct_0}{j(j+1)n(1+|x|)^{(1+\lambda)\theta}}\inf_{x \in \tilde D_{i-1}}\Pp^x\Big(\tau_{B(x,\varepsilon/2)}>t_0\Big)\\
&\ge \frac{C}{j(j+1)n(1+|x|)^{(1+\lambda)\theta}},
 \end{align*}
where the equality above is due to \eqref{l3-2-2},  in the third inequality we have used the fact that $|y-z|\le
\big((1+\lambda)\theta\log|x|\big)^{{1}/{\gamma}}$ for $y \in D_{i-1}$ and $z \in \tilde D_{i-2}$, and the last inequality follows from  Lemma \ref{l3-1}.

Hence, if $X_{\tilde \tau_{D_{i}}} \in \tilde D_{i-1}$, then we have for all $i\ge2$,
 \begin{align*}
& \Ee^{X_{\tilde \tau_{D_{i}}}}\Big(0<\tau_{D_{i-1}}<\frac{t_0}{n},X_{\tau_{D_{i-1}}}\in
\tilde D_{i-2}; \exp\Big(-\int^{\tau_{D_{i-1}}}_{0} V(X_s)\,ds\Big)\Big)\\
&\ge \sum_{j=1}^{\infty}\Ee^{X_{\tau_{D_{i}}}}\Big(\frac{t_0}{(j+1)n}\le \tau_{D_{i-1}}<\frac{t_0}{jn},X_{\tau_{D_{i-1}}}\in
\tilde D_{i-2}; \exp\Big(-\int^{\tau_{D_{i-1}}}_{0} V(X_s)\,ds\Big)\Big)\\
& \ge \sum_{j=1}^{\infty}\exp\Big(-\frac{t_0}{jn}\sup_{x \in D_{i-1}}V(x)\Big)\inf_{y \in \tilde D_{i-1}}
\Pp^y\Big(\frac{t_0}{(j+1)n}\le\tau_{D_{i-1}}<\frac{t_0}{jn}, X_{\tau_{D_{i-1}}}\in
\tilde D_{i-2}\Big)\\
&\ge \frac{C}{n(1+|x|)^{(1+\lambda)\theta}}\sum_{j=1}^{\infty}\frac{1}{j(j+1)}\exp\Big(-\frac{t_0}{jn}\sup_{x \in D_{i-1}}V(x)\Big)\\
&\ge \frac{C}{(1+|x|)^{(1+\lambda)\theta}\big(n+\sup_{x \in D_{i-1}}V(x)\big)},
 \end{align*}
where in the last inequality we have used  (\ref{rrr1}).

Furthermore, we find that (\ref{l3-3-2}) and (\ref{l3-3-3}) are still valid here. Therefore, combining with all the estimates
above, we obtain that for $|x|$ large enough
\begin{equation*}
\begin{split}
T_{t_0}^V(\I_D)(x)&  \ge C_4 \prod_{i=1}^n\Big(\frac{C_5}{|x|^{(1+\lambda)\theta}\big(n+\sup_{z \in D_{i-1}}V(z)\big)}\Big)\\
&\ge C_4\Big(\frac{C_5}{|x|^{(1+\lambda)\theta}\big(n+\sup_{|z|\le |x|+\varepsilon}V(z)\big)}\Big)^n\\
&\ge \exp\Big(-(1+\lambda)|x|\big(\theta\log|x|\big)^{-\frac{1}{\gamma}}\log(1+|x|^{(2+\lambda)\theta})-C(\theta)\Big)\\
&\ge \exp\Big(-(1+\lambda)(2+\lambda)\theta^{\frac{\gamma-1}{\gamma}}|x|\big(\log|x|\big)^{\frac{\gamma-1}{\gamma}}-C(\theta)(1+|x|)\Big),
\end{split}
\end{equation*}
where in the third inequality we have used  the property (\ref{ex2-3}).

Hence, following the same argument as that of Proposition \ref{p3-1}, we finally arrive at
\begin{equation*}
\begin{split}
\phi_1(x)\ge  \exp\Big(-(1+\lambda)(2+\lambda)\theta^{\frac{\gamma-1}{\gamma}}|x|\log^{\frac{\gamma-1}{\gamma}}(1+|x|)-
C(\theta)(1+|x|)\Big).
\end{split}
\end{equation*}
In particular, by taking $\lambda>0$ small enough in the inequality above, we indeed can get the lower bound estimate in (\ref{ex2-2})
with any constant $c_4>2$.

(3) Let
\begin{equation*}
\psi(x):=\exp\Big(-(c_0\theta)^{\frac{\gamma-1}{\gamma}}\sqrt{1+|x|^2}\log^{\frac{\gamma-1}{\gamma}}
\sqrt{1+|x|^2}\Big),
\end{equation*}
where $c_0>0$ is a constant to be determined later.

Under assumption {\bf (A3)}, by the approximation argument again it is easy to verify that $\psi \in
\D(L^V)$, we know from \eqref{ope11} that
  \begin{align*}
L^V \psi(x)&=\int_{\R^d} \Big(\psi(x+z)-\psi(x)-\langle \nabla \psi(x),z \rangle \I_{\{|z|\le 1\}}\Big)J(x,x+z)\,dz\\
&+\frac{1}{2}\int_{\{|z|\le 1\}}\langle\nabla \psi(x),
z\rangle\left(J(x,x+z)-J(x,x-z)\right)\,dz -V(x)\psi(x)\\
&\le c_1\sup_{z \in B(x,1)}\big(|\nabla^2 \psi(x)|+|\nabla \psi(x)|\big)\\
&+
\int_{\{|z|>1\}} \big(\psi(x+z)-\psi(x)\big)J(x,x+z)\,dz-V(x)\psi(x)\\
&=:I_1(x)+I_2(x)-V(x)\psi(x).
 \end{align*}

According to the proof of Theorem \ref{thm2} and the mean value theorem, for $|x|$ large enough,
\begin{equation*}
\begin{split}
I_1(x)& \le C(\theta)\log^{\frac{2(\gamma-1)}{\gamma}}(1+|x|)\exp\Big(-(c_0\theta)^{\frac{\gamma-1}{\gamma}}(|x|-1)
\log^{\frac{\gamma-1}{\gamma}}(|x|-1)\Big)\\
&\le C(\theta)\exp\Big((c_0\theta)^{\frac{\gamma-1}{\gamma}}\log^{\frac{\gamma-1}{\gamma}}(1+|x|)\Big)
\log^{\frac{2(\gamma-1)}{\gamma}}(1+|x|) \psi(x).
\end{split}
\end{equation*}

On the other hand,
\begin{equation*}
I_2(x)\le \int_{\{|z|>1\}} \psi(x+z)J(x,x+z)\,dz \le
c_2\int_{\{|z|>1\}} \psi(x+z)e^{-|z|^{\gamma}}\,dz.
\end{equation*}
For each $i\ge 1$, set $A_i:=B(0,i+1)\setminus B(0,i)$. Since
for every $x \in \R^d$ and $z \in A_i$, $$|x|-i-1\le |x+z| \le |x|+i+1,$$ we get that for any
$0<\lambda<1$ and any $x\in\R^d$ with $|x|$ large enough,
\begin{equation*}
\begin{split}
\int_{\{|z|>1\}} &\psi(x+z)e^{-|z|^{\gamma}}\,dz\\
&=
\sum_{i=1}^{\infty} \int_{A_i}\psi(x+z)e^{-|z|^{\gamma}}\,dz\\
& { \le c_3 \sum_{i=2}^{\infty}i^{d-1}e^{-(i-1)^{\gamma}}\sup_{|x|-i-1\le |z| \le |x|+i+1}
|\psi(z)|}\\
&{ \le c_4\psi(x)\sum_{i=2}^{\infty}i^{d-1}\exp\Big[\big((1+\lambda)c_0\theta
\log|x|\big)^{\frac{\gamma-1}{\gamma}}i-(i-1)^{\gamma}\Big],}
\end{split}
\end{equation*}
where in the first inequality we have used the fact that $|A_i|\le Ci^{d-1}$, and the second inequality follows from the fact that
for $|x|$ large enough $$\sup_{|z|\ge |x|-i-1}
|\psi(z)|\le \exp\Big[\big((1+\lambda)c_0\theta\log|x|\big)^{\frac{\gamma-1}{\gamma}}i\Big]\psi(x),$$ thanks to again the mean value theorem.

Furthermore, set $N(x):=\Big\lfloor
(6(1+\lambda)c_0\theta\log|x|)^{\frac{1}{\gamma}}\Big\rfloor+1$. Noticing that
for every $i>N(x)$,
$$\big((1+ \lambda)c_0\theta\log|x|\big)^{\frac{\gamma-1}{\gamma}}i-(i-1)^{\gamma}\le -\frac{i^{\gamma}}{2},$$
we have
 \begin{align*}
&\sum_{i=2}^{\infty}i^{d-1}\exp\left[\left(\left(1+
\lambda\right)c_0\theta\log|x|\right)^{\frac{\gamma-1}{\gamma}}i-(i-1)^{\gamma}
\right]\\
&=\sum_{i=2}^{N(x)}i^{d-1}\exp\left[\left(\left(1+ \lambda\right)c_0\theta\log|x|\right)^{\frac{\gamma-1}{\gamma}}i-
(i-1)^{\gamma}
\right]\\
&\quad +\sum_{i=N(x)+1}^{\infty}i^{d-1}\exp\left[\left(\left(1+ \lambda\right)c_0\theta\log|x|\right)^{\frac{\gamma-1}{\gamma}}i-(i-1)^{\gamma}
\right]\\
&\le N(x)^d\exp\left[\sup_{s \in \R}
\left(\left(\left(1+\lambda\right)c_0\theta\log|x|\right)^{\frac{\gamma-1}{\gamma}}s-(s-1)^{\gamma}\right)\right]+
\sum_{i=N(x)+1}^{\infty}i^{d-1}e^{-\frac{i^{\gamma}}{2}}\\
&\le C(\theta)
(\theta\log|x|)^{\frac{d}{\gamma}}\exp\left((1+2\lambda)c_0\theta\left(\left(\frac{1}{\gamma}\right)^{\frac{1}{\gamma-1}}
-\left(\frac{1}{\gamma}\right)^{\frac{\gamma}{\gamma-1}}\right)\log|x|\right)\\
&\quad+ \exp\bigg[{-\frac{(1+\lambda)c_0\theta }{4}}\log|x|\bigg] ,
 \end{align*}

where in the last inequality we have used the facts that for $|x|$ large enough,
{ \begin{equation*}
\begin{split}
& \sup_{s \in \R}
\left\{\left(\left(1+ \lambda\right)c_0\theta\log|x|\right)^{\frac{\gamma-1}{\gamma}}s-(s-1)^{\gamma}\right\}\\
&\le
(1+2\lambda)c_0\theta\left(\left(\frac{1}{\gamma}\right)^{\frac{1}{\gamma-1}}
-\left(\frac{1}{\gamma}\right)^{\frac{\gamma}{\gamma-1}}\right)\log|x|
\end{split}
\end{equation*}
}
and $$\sum_{i=n}^{\infty}i^{d-1}e^{-\frac{i^{\gamma}}{2}} \le e^{-\frac{n^{\gamma}}{4}}$$ for $n$ large enough.

Combining all the estimates above and taking $$c_0=
\frac{1}{2(1+2\lambda)
\left(\left(\frac{1}{\gamma}\right)^{\frac{1}{\gamma-1}}
-\left(\frac{1}{\gamma}\right)^{\frac{\gamma}{\gamma-1}}\right)},$$ we derive that
for $|x|$ large enough,
\begin{equation*}
\begin{split}
L^V (x)\psi (x)&\le C(\theta)|x|^{\frac{2}{3}\theta}\psi(x)-V(x)\psi(x)\le 0,
\end{split}
\end{equation*}
which implies that
\begin{equation*}
L^V \psi(x) \le C(\theta) \psi(x), \ \ x \in \R^d
\end{equation*}
for some constant $C(\theta)>0$. Therefore, according to Proposition \ref{pro--00},  we can obtain the desired upper bound estimate for $\phi_1$ as that in
(\ref{ex2-2}).
\end{proof}

\ \

At the last, we turn to the proofs of Theorem $\ref{thm3}$ and Example \ref{ex3}.

\begin{proof}[Proof of Theorem $\ref{thm3}$]
(1) Following the argument of  Theorem \ref{thm2}, for $R$, $r>0$ large
enough
$$
\Phi(R)\ge { C}R\log^{\theta_1}R$$ and
$$ \beta_\Phi^{-1}(r)\le \frac{C}{\log(1+
r)\log^{\theta_1-1}\log(e+r)}.$$
Hence, for $s>0$ small enough,
\begin{equation*}
\gamma(s)=\Theta\left(\Phi^{-1}\left(\frac{2}{s}\right)\right)
\le C_1\exp\left(-{ {C_2}c_6}\left(\frac{1}{s}\right)^{{\eta_1}}
\left(\log \frac{1}{s}\right)^{{\eta_2-\eta_1\theta_1}}\right).
\end{equation*}
Then,  for any fixed $\delta>e$, there is an integer $N_0(\delta)\ge1$ such that for all
$n\ge N_0(\delta)$,
\begin{equation*}
\begin{split}
& s_n:=\beta_\Phi^{-1}\left(c\delta^n\right)\le C_3(\log\delta)^{-1} n^{-1}
(\log n)^{-(\theta_1-1)},
\end{split}
\end{equation*}
where $C_3>0$ is independent of $\delta$. Therefore, for $n$ large enough,
\begin{equation}\label{thm3-1}
\gamma(s_n)\le C\exp\left[-c_6C_4\left(\log \delta\right)^{\eta_1}
n^{\eta_1}
\left(\log n\right)^{\eta_2-\eta_1}\right],
\end{equation}
where {$C_4$ } is a constant independent of $\delta$.

According to (\ref{thm3-1}),  if
$\eta_1=1$ and $\eta_2>1$, then (\ref{t3-1-4}) holds true, and we have the following estimate for the rate function
$\tilde \beta_\Phi(s)$ defined by (\ref{t3-1-5a})
\begin{equation}\label{thm3-2}
\tilde \beta_\Phi(s)\le C\exp\left[C\Big(1+\frac{1}{s}\Big)\left(\log^{1-{\theta_1}}\Big(1+ \frac{1}{s}\Big)\right)\right],
\end{equation}
which implies that (\ref{t3-1-5}) is satisfied. Therefore, according to
Theorem \ref{t3-1}(3),  we know that the semigroup $(T_t^V)_{t \ge 0}$ is intrinsically ultracontractive.

On the other hand, it follows from (\ref{thm3-1}) that, when $\eta_1=\eta_2=1$ and {
$c_6>\frac{1}{C_4}$,} (\ref{t3-1-4}) holds true. Then, following the arguments above, we can get the same estimate (\ref{thm3-2}) for
$\tilde \beta_\Phi(s)$ (possibly with different constant $C$ in (\ref{thm3-2})), and
so the semigroup $(T_t^V)_{t \ge 0}$ is intrinsically ultracontractive. The proof of the first assertions is complete.

(2) When $d>\alpha_1$, there exists a constant $C>0$ such that
for $R$ large enough
\begin{equation*}
\Psi(R)\le C\left(\frac{1}{R\log^{\theta_1}R}+\frac{1}{R\log^{{\alpha_1\eta_3}/{d}}R}\right).
\end{equation*}
Then, following the same argument of part (1), we can arrive at the second conclusion by Theorem \ref{t3-1}(2).
\end{proof}
\begin{proof}[Proof of Example $\ref{ex3}$]
Here we still try to disprove \cite[Condition 1.3, p.\ 5027]{KS}. Let
$D=B(0,1)$ and $t=1$. According to the proof of Example \ref{ex2}, for all $n$ large enough,
\begin{equation}\label{ex3-1}
T_1^V(\I_D)(x_n)\le C_1\exp\Big(-C_2|x_n|\log|x_n|\Big)=
C_1\exp\Big(-C_3n^{k_0}\log n\Big).
\end{equation}
On the other hand, for $n$ large enough
\begin{equation*}
\begin{split}
T_1^V(\I_{B(x_{n},1)})(x_{n})&\ge T_1^V(\I_{B(x_{n},r_{n})})(x_{n})\\
& \ge \Ee^{x_{n}}\Big(\tau_{B(x_{n},r_{n})}>1; \exp\Big(-\int_0^1 V(X_s)ds\Big)\Big)\\
& =e^{-1}\Pp^{x_{n}}(\tau_{B(x_{n},r_{n})}>1)\\
&=e^{-1}\Pp^{0}(\tau_{B(0,r_n)}>1),
\end{split}
\end{equation*}
where the first equality follows from the fact that $V(z)=1$ for
every $z \in B_n$, and in the last equality we have used the space-homogeneous property of truncated symmetric $\alpha$-stable process.

Let $(X_t)_{\ge0}$ be the truncated symmetric $\alpha$-stable process. By the Meyer construction for truncated $\alpha$-stable process (see \cite[Remark 3.5]{BBCK}), there corresponds to a symmetric $\alpha$-stable process $(X_t^*)_{t\ge0}$ (on the same probability space), such that
$X_t=X_t^*$ for any $t\in(0, N_1^*)$, where
$$
N^*_1=\inf\big\{t\ge0: |\Delta X^*_t|>1\big\},$$
and $\Delta X^*_t:=X^*_t-X^*_{t-}$ denotes the jump of $(X^*_t)_{t\ge0}$ at time $t$.
In the following, let $$\tau^*_{B(0,r)}=\inf\big\{t>0:X_t^*\notin B(0,r)\big\}$$ be the first exit time from $B(0,r)$ of the process $(X_t^*)_{t\ge0}$.
Note that, under $\Pp^0$ the event
$\{X_t^* \in B(0,r),\,\forall\ t\in [0,1]\}$ implies that the process $(X_t^*)_{t\ge0}$ does not have any jump bigger than $1$.
Then we find that there are constants $C_4, \lambda^*_1>0$ such that for all $n\ge1$ large enough, \begin{equation*}
\begin{split}
\Pp^{0}(\tau_{B(0,r_{n})}>1)&\ge \Pp^{0}(\tau^*_{B(0,r_{n})}>1)\\
&=\Pp^{0}(\tau^*_{B(0,1)}>r_{n}^{-\alpha})\\
&=\int_{B(0,1)} p^*_{B(0,1)}(r_n^{-\alpha},0,z)\,dz\\
&\ge C_4 e^{-\lambda^*_1 r^{-\alpha}_{n}}.
\end{split}
\end{equation*}
Here in the first equality we have used the scaling property of symmetric $\alpha$-stable process, in the second equality  $p^*_{B(0,1)}(t,x,y)$ denotes the Dirichlet heat kernel of the process $(X_t^*)_{t\ge0}$ on $B(0,1)$, and the last inequality follows from lower bound of $p^*_{B(0,1)}(t,x,y)$ established in \cite[Theorem 1.1(ii)]{CKS}.
 Hence, for $n$ large enough,
\begin{equation}\label{ex3-2}
\begin{split}
T_1^V(\I_{B(x_n,1)})(x_n)&\ge C_4 \exp\Big(-\lambda_1 n^{k_0-\frac{\alpha}{d}}\Big).
\end{split}
\end{equation}

According to (\ref{ex3-1}) and (\ref{ex3-2}) above,
we know that for any constant $C>0$, the following inequality
\begin{equation*}\label{ex3-3}
T_1^V(\I_{B(x,1)})(x)\le C T_1^V(\I_D)(x).
\end{equation*}
does not hold for $x=x_{n}$ when $n$ large enough.
In particular,
\cite[Condition 1.3, p.\ 5027]{KS} is not satisfied, and so the semigroup $(T_t^V)_{t \ge 0}$ is not intrinsically ultracontractive.

However, for every $R\ge 2$ and $n\ge 1$ with $n^{k_0}\le R \le (n+1)^{k_0}$,
\begin{equation*}
\begin{split}
|\{x \in \R^d: x \in A, |x|\ge R\}|&\le
\sum_{m=n}^{\infty}|B(x_m,r_m)|=
\sum_{m=n}^{\infty}r_m^d\\
&= \sum_{m=n}^{\infty} m^{-\frac{dk_0}{\alpha}+1}
\le C_5n^{-\frac{dk_0}{\alpha}+2}\\
&\le C_6 \Big((n+1)^{k_0}\Big)^{-\frac{d}{\alpha}+\frac{2}{k_0}}\le
\frac{C_6}{ R^{\frac{d}{\alpha}-\varepsilon}}
\end{split}
\end{equation*}
holds for some constant $C_6$ independent of $R$, where
in the last inequality we have used the fact that $\frac{2}{k_0}<\varepsilon$. Therefore
(\ref{ex3-0}) holds true. By now we have finished the proof.
\end{proof}
\section{Appendix}
In this appendix, we will present the proofs of Propositions \ref{p1-1}
and \ref{p1-2}.

\begin{proof}[Proof of Proposition $\ref{p1-1}$]
Let $(T_t)_{t\ge0}$ be the Markov semigroup associated with the regular
Dirichlet form $(D, \mathscr{D}(D))$. Under assumption {\bf (A1)},
for every $t>0$,
$$\|T_t\|_{L^1(\R^d;dx)\to L^\infty(\R^d; dx)}=\sup_{x,y \in \R^d}p(t,x,y)\le
c_t.$$
 According to \cite[Theorem 3.3.15]{WBook}, the following super Poincar\'{e} inequality holds
\begin{equation*}\int f^2(x)\,dx\le rD(f,f)+\beta(r)\Big(\int |f(x)|\,dx\Big)^2,\quad r>0, f\in \mathscr{D}(D),\end{equation*} where
$$\beta(r)=\inf_{s\le r, t>0} \frac{s \|T_t\|_{L^1(\R^d;dx)\to L^\infty(\R^d; dx)}}{t} \exp\Big(\frac{t}{s}-1\Big)\le \|T_r\|_{L^1(\R^d;dx)\to L^\infty(\R^d; dx)}\le c_r.$$
Therefore, we can take the reference symmetric function $\mu$ in \cite[(1.2)]{WW08} to be Lebesgue measure.

Clearly, the potential function $V$ satisfies \cite[(1.3) and (1.5)]{WW08}. Then, the desired assertion immediately follows from \cite[Corollary 1.3]{WW08}.
\end{proof}

\begin{proof}[Proof of Proposition $\ref{p1-2}$]
For any $t>0$ and $x,y\in\R^d$, $p^V(t,x,y)\le p(t,x,y)\le c_t$, so
the semigroup $(T_t^V)_{t \ge 0}$ is ultracontractive. In
particular, by the symmetric property of $(T_t^V)_{t \ge 0}$ on
$L^2(\R^d;dx)$, we know that $\|T_t^V\|_{L^2(\R^d;dx)\to
L^\infty(\R^d; dx)}<\infty.$ This, along with $T_t^V
\phi_1=e^{-\lambda_1 t}\phi_1$ and $\phi_1 \in L^2(\R^d;dx)$, yields
that there is a version of $\phi_1$ which is bounded.

For any $R>0$, let $\phi_1^R(x):=e^{-\lambda_1 t}\int_{\{|y|\le R\}}
p^V(t,x,y)\phi_1(y)\,dy$. For any $y\in\R^d$ and $t>0$, the function
$x\mapsto p_t^V(x,y)$ is continuous and $p^V(t,x,y)\le p(t,x,y)\le
c_t$. According to the fact that $\phi_1$ is locally
$L^1(\R^d;dx)$-integrable  and the dominated convergence theorem,
$\phi_1^R$ is also a continuous function. Now, for every fixed $x_0
\in \R^d$, let $\{x_n\}_{n=1}^{\infty}\subseteq \R^d$ be a sequence
such that $\lim_{n \rightarrow \infty}x_n=x_0$. Then we obtain
 $$\aligned
&|\phi_1(x_n)-\phi_1(x_0)|\\
&=e^{-\lambda_1 t}
|T^V_t\phi_1(x_n)-T^V_t\phi_1(x_0)|\\
&=e^{-\lambda_1 t}\Big|\int_{\R^d}p^V(t,x_n,y)\phi_1(y)\,dy-
\int_{\R^d}p^V(t,x_0,y)\phi_1(y)\,dy\Big|\\
&\le  |\phi_1^R(x_n)-\phi_1^R(x_0)|+
2\sup_{x \in \R^d}\Big(\int_{\R^d}\big(p^V(t,x,y)\big)^2\,dy\Big)^{{1}/{2}}
\Big(\int_{\{|y|>R\}}\phi_1^2(y)\,dy\Big)^{{1}/{2}}\\
&\le |\phi_1^R(x_n)-\phi_1^R(x_0)|+2\sqrt{c_t}\Big(\int_{\{|y|>R\}}\phi_1^2(y)\,dy\Big)^{{1}/{2}}.
\endaligned$$
Letting $n \rightarrow \infty$ and then $R \rightarrow \infty$, we
arrive at $\lim_{n \rightarrow \infty}\phi_1(x_n)=\phi_1(x)$.
Therefore there exists a version of $\phi_1$ which is continuous.

Let $(D^V, \mathscr{D}(D^V))$ be the Dirichlet form associated with $(T_t^V)_{t \ge 0}$. Due to the following variational principle
\begin{equation*}
\lambda_1=\inf\bigg\{\frac{D^V(f,f)}{\int_{\R^d} f^2(x)\,dx}: f \in
\D(D^V),\ f \neq 0\bigg\}= D^V(\phi_1,\phi_1),
\end{equation*}
and the fact $D^V(|\phi_1|,|\phi_1|)\le D^V(\phi_1,\phi_1)$, we know
that $\phi_1 \ge 0$. Now, assume that $\phi_1(x_0)=0$ for some $x_0
\in \R^d$. Since $p^V(t,x,y)>0$ for any $t>0$ and $x,y\in\R^d$, and
$$\phi_1(x_0)=e^{-\lambda_1
t}\int_{\R^d}p^V(t,x_0,y)\phi_1(y)\,dy=0,$$
 we find by the continuity of $\phi_1$ that $\phi_1(x)=0$ for every $x \in \R^d$. This contradiction implies that
$\phi_1>0$ is positive everywhere. The proof is complete.
\end{proof}



\begin{thebibliography}{99}

\bibitem{BBCK} Barlow, M.T., Bass, R.F., Chen, Z.-Q. and Kassmann, M.:
Non-local Dirichlet forms and symmetric jump processes,
\emph{Trans. Amer. Math. Soc.} \textbf{361} (2009), 1963--1999.

\bibitem{BKK} Bass, R.F., Kassmann, M. and Kumagai, T.:
Symmetric jump processes: localization, heat kernels, and convergence,
\emph{Ann. Inst. H. Poincar\'e--Probabilit\'es et Statistiques}
 \textbf{46} (2010), 59--71.

\bibitem{CW14}
Chen, X. and Wang, J.: Functional inequalities for nonlocal Dirichlet forms with finite
range jumps or large jumps, \emph{Stoch. Proc. Appl.} \textbf{124} (2014), 123--153.



\bibitem{CK}
Chen, Z.-Q. and Kumagai, T.: Heat kernel estimates for stable-like processes on $d$-sets,
\emph{Stoch. Proc. Appl.} \textbf{108} (2003), 27--62.

\bibitem{CK1} Chen, Z.-Q. and Kumagai, T.:
Heat kernel estimates for jump processes of mixed types on metric measure spaces,
\emph{Probab. Theory Relat. Fields} \textbf{140} (2008), 277--317.




\bibitem{CKK2} Chen, Z.-Q., Kim, P. and Kumagai, T.:
Weighted Poincar\'e inequality and heat kernel estimates for finite range jump processes,
\emph{Math. Ann.} \textbf{342} (2008), 833--883.


\bibitem{CKK1} Chen, Z.-Q., Kim, P. and Kumagai, T.:
Global heat kernel estimates for symmetric jump processes,
\emph{Trans. Amer. Math. Soc.} \textbf{363} (2011), 5021--5055.

\bibitem{CKS} Chen, Z.-Q., Kim, P. and Song, R.:
Heat kernel estimates for Dirichlet fractional Laplacian,
\emph{J. Euro. Math. Soc.} \textbf{12} (2010), 1307--1329.

\bibitem{CS} Chen, Z.-Q. and Song, R.: Intrinsic ultracontractivity and conditional gauge for symmetric
stable processes, \emph{J. Funct. Anal.} \textbf{150} (1997), 204--239.


\bibitem{CZ} Chung, K.L. and Zhao. Z.: \emph{From Brownian Motion to Schr\"odinger's equation}, Springer, New
York, 1995.


\bibitem{DS} Davies, E.\ B. and
Simon, B.: Ultracontractivity and heat kernels for Schr\"odinger operators and Dirichlet Laplacians,
\emph{J. Funct. Anal.} \textbf{59} (1984), 335--395.

\bibitem{FOT}
Fukushima, M., Oshima, Y. and Takeda, M.: \emph{Dirichlet Forms and Symmetric Markov Processes}, de Gruyter, Berlin 2011, 2nd.

\bibitem{KK}
Kaleta, K and Kulczycki, T.:
Intrinsic ultracontractivity for Schr\"odinger operators based on fractional Laplacians,
\emph{Potential Anal.} \textbf{33} (2010), 313--339.

\bibitem{KL}
Kaleta, K. and L\H{o}rinczi, J.: Pointwise eigenfunction estimates
and intrinsic ultracontractivity-type properties of Feynman-Kac
semigroups for a class of L\'{e}vy processes, to appear in
\emph{Ann. Probab.} 2013, also see arXiv:1209.4220v2



\bibitem{KS}
Kulczycki, T. and Siddeja, B.: Intrinsic ultracontractivity of the Feynman-Kac semigroup for relativistic stable processes, \emph{Trans. Am. Math. Soc.} \textbf{358} (2006), 5025--5057.


\bibitem{Wang00} Wang, F.-Y.: {Functional inequalities for empty spectrum estimates,}
 \emph{J. Funct. Anal.\ } \textbf{170} (2000), 219--245.

\bibitem{W2}Wang, F.-Y.: Functional inequalities, semigroup properties and spectrum estimates,
\emph{Infin. Dimens. Anal. Quant. Probab. Relat. Top.} \textbf{3} (2000), 263-295.


\bibitem{Wang02} Wang, F.-Y.: {Functional inequalities and spectrum estimates: the infinite measure case,}
 \emph{J. Funct. Anal.\ } \textbf{194} (2002), 288--310.

\bibitem{WBook}
Wang, F.-Y.: \emph{Functional Inequalities, Markov Processes and
Spectral Theory}, Science Press, Beijing 2005.

\bibitem{WW}
Wang, F.-Y. and Wang, J.: Functional inequalities for stable-like
Dirichlet forms, to appear in \emph{J.\ Theor.\ Probab.} 2013, also
see arXiv:1205.4508v3

\bibitem{WW08}
Wang, F.-Y. and Wu, J.-L.: Compactness of Schr\"{o}dinger semigroups
with unbounded below potentials, \emph{Bull. Sci. Math.}
\textbf{132} (2008), 679--689.

\bibitem{W2009}
Wang, J.: Symmetric L\'{e}vy type operator, \emph{Acta Math. Sin.
Eng. Ser.} \textbf{25} (2009), 39--46.


\bibitem{WJ13}
Wang, J.:  A simple approach to functional inequalities for
non-local Dirichlet forms, \emph{ESAIM: Probab.
Statist.} \textbf{18} (2014), 503--513.
\end{thebibliography}
\end{document}